\documentclass[reqno,11pt]{amsart}

\usepackage{a4wide}
\usepackage{color}
\usepackage{mathrsfs}
\usepackage{mathtools}
\usepackage{amsmath}
\usepackage{amssymb}
\usepackage{bbm}
\usepackage{esint}
\usepackage{nicefrac}
\numberwithin{equation}{section}
\usepackage[colorlinks,citecolor=green,linkcolor=red]{hyperref}

\usepackage[latin1]{inputenc}

\newcommand{\N}{\mathbb{N}}
\newcommand{\R}{\mathbb{R}}

\newcommand{\sfd}{{\sf d}}
\renewcommand{\d}{{\mathrm d}}
\newcommand{\e}{{\rm e}}
\newcommand{\X}{{\rm X}}

\newcommand{\mm}{\mathfrak{m}}
\newcommand{\1}{\mathbbm 1}

\newcommand{\LIP}{{\rm LIP}}
\newcommand{\Lip}{{\rm Lip}}
\newcommand{\lip}{{\rm lip}}

\newcommand{\ppi}{{\mbox{\boldmath\(\pi\)}}}
\newcommand{\eeta}{{\mbox{\boldmath\(\eta\)}}}
\newcommand{\sppi}{{\mbox{\scriptsize\boldmath\(\pi\)}}}
\newcommand{\seeta}{{\mbox{\scriptsize\boldmath\(\eta\)}}}

\newcommand{\limi}{\varliminf}
\newcommand{\lims}{\varlimsup}

\newcommand{\fr}{\penalty-20\null\hfill\(\blacksquare\)}

\newtheorem{theorem}{Theorem}[section]
\newtheorem{corollary}[theorem]{Corollary}
\newtheorem{lemma}[theorem]{Lemma}
\newtheorem{proposition}[theorem]{Proposition}
\newtheorem{definition}[theorem]{Definition}

\newtheorem{remark}[theorem]{Remark}

\title{On master test plans for the space of BV functions}

\author[Francesco Nobili]{Francesco Nobili}\address[Francesco Nobili]
{SISSA, Via Bonomea 265, 34136 Trieste, Italy.}
\email{fnobili@sissa.it}

\author[Enrico Pasqualetto]{Enrico Pasqualetto}
\address[Enrico Pasqualetto]{Scuola Normale Superiore, Piazza dei
Cavalieri, 7, 56126 Pisa, Italy.}
\email{enrico.pasqualetto@sns.it}

\author[Timo Schultz]{Timo Schultz}\address[Timo Schultz]
{Fakultat f\"{u}r Mathematik, Universit\"{a}t Bielefeld,
Postfach 100131, 33501 Bielefeld, Germany, {\sl and,}
Institute for Applied Mathematics, University of Bonn, Endenicher Allee 60, 53115 Bonn, Germany
}
\email{tschultz@math.uni-bielefeld.de}
\date{\today}
\begin{document}
\begin{abstract}
We prove that on an arbitrary metric measure space a countable
collection of test plans is sufficient to recover all \(\rm BV\)
functions and their total variation measures. In the setting of
non-branching \({\sf CD}(K,N)\) spaces (with finite reference
measure), we can additionally require these test plans to be
concentrated on geodesics.
\end{abstract}
\maketitle
\allowdisplaybreaks
\tableofcontents
\section{Introduction}
The first notion of a \emph{function of bounded variation}
(or just a \emph{BV function} for short) in the setting of metric
measure spaces dates back to a paper by M.\ Miranda Jr., published
almost 20 years ago \cite{Miranda03}. Since then, several equivalent
definitions of BV function have been introduced and studied in the
literature. The most relevant one for the purposes of the present
paper is the one proposed by L.\ Ambrosio and
S.\ Di Marino in \cite{Ambrosio-DiMarino14}, which we are going to
describe informally. 
While the original approach in \cite{Miranda03}
is of `Eulerian' nature as it is based upon a relaxation procedure
using Lipschitz functions, the definition in \cite{Ambrosio-DiMarino14}
is `Lagrangian' as it ultimately looks at the behaviour of functions
along well-chosen curves. 
The motivation of the latter approach comes from a well established BV-theory in the classical Euclidean space. In this standard framework, BV functions are known to behave well under one dimensional restrictions (see \cite{AmbrosioFuscoPallara} for a thorough discussion). In particular, by the use of \emph{slicing} techniques, it is not only possible to see that a function is of bounded variation if and only if its restriction to all lines parallel to the coordinate axes is of bounded variation, but it is also possible to reconstruct the \emph{full} total variation as a superposition of one dimensional variations of the sliced functions. In the absence of smooth coordinates, the key concept to mimic this characterisation is that
of a test plan, originally introduced in \cite{AmbrosioGigliSavare11,
AmbrosioGigliSavare11-3}. Given an arbitrary metric measure space
\((\X,\sfd,\mm)\), we say that a Borel probability measure \(\ppi\)
on the space of continuous curves \(C([0,1],\X)\) is a
\emph{\(\infty\)-test plan} (cf.\ Definition \ref{def:test_plan})
provided it is concentrated on \(\rm L\)-Lipschitz curves (for some
constant \({\rm L}\geq 0\)) and it satisfies
\((\e_t)_\#\ppi\leq{\rm C}\mm\) for all \(t\in[0,1]\),
where \({\rm C}>0\) is a constant independent of \(t\), while
\(\e_t\colon C([0,1],\X)\to\X\) stands for the evaluation map
\(\gamma\mapsto\gamma_t\). The quantities \({\rm Lip}(\ppi)\) and
\({\rm Comp}(\ppi)\) are then defined as the minimal such \(\rm L\)
and \(\rm C\), respectively. At this point, for any given family
\(\Pi\) of \(\infty\)-test plans we declare (cf.\ Definition
\ref{def:BV}) that a function \(f\in L^1(\mm)\)
\emph{has bounded \(\Pi\)-variation} (\(f\in{\rm BV}_\Pi(\X)\)
for short) if there exists a finite Borel measure \(\mu\) on \(\X\)
satisfying (in the sense of measures) the following inequality:
\begin{equation}\label{eq:intro_BV}
\int\gamma_\#|D(f\circ\gamma)|\,\d\ppi(\gamma)
\leq{\rm Comp}(\ppi){\rm Lip}(\ppi)\mu,\quad\text{ for every }
\ppi\in\Pi.
\end{equation}
The minimal such measure \(\mu\) is denoted by
\(|{\boldsymbol D}f|_\Pi\) and called the \emph{total \(\Pi\)-variation
measure} of \(f\). When \(\Pi\) is the family of all
\(\infty\)-test plans, we recover the notion of BV space
\({\rm BV}(\X)\) introduced in \cite{Ambrosio-DiMarino14}.
This choice of terminology supports the goal of the present paper: roughly speaking, we aim at reducing the number of \(\infty\)-test plans needed in \eqref{eq:intro_BV} to
detect both a BV function and its total variation measures. More
specifically, the two main results of the paper are the following ones:
\begin{itemize}
\item[\(\rm i)\)] On any metric measure space, there exists
a countable family \(\Pi\) of \(\infty\)-test plans such that
\({\rm BV}_\Pi(\X)={\rm BV}(\X)\) and \(|{\boldsymbol D}f|_\Pi=
|{\boldsymbol D}f|\) for every \(f\in{\rm BV}(\X)\). For brevity, we
just say that \(\Pi\) is a \emph{master family} for \({\rm BV}(\X)\).
\item[\(\rm ii)\)] On \({\sf CD}(K,N)\) metric measure
spaces (whose reference measure has finite mass) that are also \emph{non-branching} (see Definition \ref{def:CDq} and \eqref{eq:nonbranching} below), the family of those
\(\infty\)-test plans that are concentrated on geodesics constitutes
a master family for \({\rm BV}(\X)\).
\end{itemize}
Before passing to a more detailed description of i) and ii),
we mention other two achievements:
\begin{itemize}
\item[\(\rm iii)\)] We introduce yet another notion of BV space,
which we call the \emph{curvewise BV space}. We prove that (on
arbitrary metric measure spaces) this new notion is equivalent
to the others, and i) yields existence of a single master test
plan in the curvewise sense.
\item[\(\rm iv)\)] In the setting of \({\sf RCD}(K,N)\) metric measure
spaces, we construct a test plan which is concentrated on geodesics
and is a master plan for the \(1\)-Sobolev space \(W^{1,1}(\X)\).
\end{itemize}
The result stated in i) is the content of Theorem
\ref{thm:countable_master_tp}. Its proof argument is rather
easy and relies on the (a priori) weaker notion of a
\emph{Beppo Levi BV function}, introduced by Di Marino in his
PhD thesis \cite{DiMarinoPhD}. Given a family \(\Pi\) of
\(\infty\)-test plans, we declare that some function \(f\in L^1(\mm)\)
belongs to the space \({\rm BV}_\Pi^*(\X)\) provided there exists
a constant \(C\geq 0\) such that
\begin{equation}\label{eq:intro_BV_star}
\int f(\gamma_1)-f(\gamma_0)\,\d\ppi(\gamma)\leq
{\rm Comp}(\ppi){\rm Lip}(\ppi)C,\quad\text{ for every }\ppi\in\Pi.
\end{equation}
We introduce this auxiliary definition in the paragraph preceding
Theorem \ref{thm:master_G}. When \(\Pi\) is the family of all
\(\infty\)-test plans, we recover the Beppo Levi BV space
\({\rm BV}^\star(\X)\) proposed in \cite{DiMarinoPhD}
(denoted by \(BV_{BL}(\X)\) therein). The fact that
\({\rm BV}(\X)\subseteq{\rm BV}^\star(\X)\) (and more generally
that \({\rm BV}_\Pi(\X)\subseteq{\rm BV}_\Pi^\star(\X)\) for any
plan family \(\Pi\)) can be easily checked, but the converse
inclusion is a highly non-trivial result obtained in
\cite{Ambrosio-DiMarino14,DiMarinoPhD}, which we will report
in Theorem \ref{thm:equiv_BV}. The prime advantage of working
with the space \({\rm BV}^\star(\X)\) is that (for \(f\)
sufficiently regular) the left-hand side in the inequality
\eqref{eq:intro_BV_star} is continuous under narrow convergence
of \(\ppi\). Combining this property with the separability of the
space of \(\infty\)-test plans and the identity
\({\rm BV}(\X)={\rm BV}^\star(\X)\), we can eventually obtain
in Theorem \ref{thm:countable_master_tp} a countable family of
master plans for \({\rm BV}(\X)\).
\medskip

The verification of ii) is more involved and based on techniques coming from optimal transportation (see \cite{Villani09}), so let us
spend a few words about the strategy behind its proof.
The \emph{Curvature-Dimension condition} \({\sf CD}(K,N)\),
pioneered by Lott--Sturm--Villani
\cite{Lott-Villani09,Sturm06I,Sturm06II},
provides an intrinsic way to impose a lower bound \({\rm Ric}\geq K\)
on the Ricci curvature and an upper bound \({\rm dim}\leq N\) on the
dimension, when dealing with (possibly very irregular) metric measure
spaces. The Curvature-Dimension condition is expressed via the
convexity properties of the R\'{e}nyi entropy functionals in the
space \(\mathscr P_q(\X)\) of Borel probability measures with finite
\(q^{th}\)-moment, equipped with the \emph{Wasserstein distance}
\(W_q\). The original definition of a \({\sf CD}(K,N)={\sf CD}_2(K,N)\)
space in fact corresponds to the exponent \(q=2\), but later on
(see \cite{Kell17b}) also \({\sf CD}_q(K,N)\) spaces for different
\(q\in(1,\infty)\) have been studied. In the case of
non-branching \({\sf CD}(K,N)\) spaces (whose reference
measure \(\mm\) is finite), we know from the results in \cite{ACCMCS20}
that \({\sf CD}_q(K,N)\) is actually independent of \(q\).  
Another important property we will need
is the existence of \emph{\(q\)-optimal dynamical plans}:
given a \({\sf CD}_q(K,N)\) space \((\X,\sfd,\mm)\) and measures
\(\mu_0,\mu_1\in\mathscr P_q(\X)\) with \(\mu_0,\mu_1\ll\mm\), we can
find a Borel probability measure \(\ppi\) on the set \({\rm Geo}(\X)\)
of all geodesics, such that \([0,1]\ni t\mapsto(\e_t)_\#\ppi\) is a
geodesic in the Wasserstein space \(\big(\mathscr P_q(\X),W_q\big)\)
joining \(\mu_0\) and \(\mu_1\). We now briefly explain how these two
properties -- the fact that \({\sf CD}_q(K,N)\) does not depend on
\(q\) and the existence of \(q\)-optimal dynamical plans -- enter into
play during the proof of Theorem \ref{thm:master_G}, which states that
on non-branching \({\sf CD}(K,N)\) spaces (with \(\mm(\X)<+\infty\))
those \(\infty\)-test plans that are concentrated on \({\rm Geo}(\X)\)
form a master family for \({\rm BV}(\X)\). For the sake of exposition,
we first consider the case \(K\geq 0\), which is technically
simpler to handle:
\begin{itemize}
\item The first step is to show that, given any boundedly-supported
\(\infty\)-test plan \(\eeta\), we can find a \(\infty\)-optimal
dynamical plan \(\ppi\) with \((\e_0)_\#\ppi=(\e_0)_\#\eeta\),
\((\e_1)_\#\ppi=(\e_1)_\#\eeta\), and
\begin{equation}\label{eq:intro_BIP}
{\rm Comp}(\ppi)\leq{\rm Comp}(\eeta).
\end{equation}
We will prove this fundamental property in Theorem
\ref{thm:CD_are_BIP}. First, by suitably adapting known results
we cook up, for any \(q\in(1,\infty)\), a \(q\)-optimal dynamical
plan \(\ppi_q\) \emph{interpolating} \(\eeta\) (namely,
\((\e_i)_\#\ppi_q=(\e_i)_\#\eeta\) for \(i=0,1\)) and verifying
\eqref{eq:intro_BIP}. Then, by using a standard compactness argument
(which builds upon Proposition \ref{prop:compactness_tp}) we obtain
the desired \(\infty\)-optimal dynamical plan \(\ppi\) as a narrow
sublimit of \(\ppi_q\) as \(q\to\infty\).
\item As pointed out in Remark \ref{rmk:poly_interp_ineq}, since
\(\ppi\) interpolates \(\eeta\), one also has \({\rm Lip}(\ppi)
\leq{\rm Lip}(\eeta)\).
\item Called \(\Pi_{\rm Geo}\) the set of \(\infty\)-test
plans concentrated on \({\rm Geo}(\X)\) and given
\(f\in{\rm BV}_{\Pi_{\rm Geo}}^\star(\X)\), we aim to prove that
\(f\in{\rm BV}(\X)\) and \(|{\boldsymbol D}f|(\X)\leq
|{\boldsymbol D}f|_{\Pi_{\rm Geo}}^\star(\X)\).
This would imply that \(\Pi_{\rm Geo}\) is a master family for
\({\rm BV}(\X)\), as the other implications can be easily
established. We can argue as follows: to any boundedly-supported
\(\infty\)-test plan \(\eeta\) we associate its interpolating
plan \(\eeta\in\Pi_{\rm Geo}\) as above, thus we may estimate
\[\begin{split}
\int f(\gamma_1)-f(\gamma_0)\,\d\eeta(\gamma)
&=\int f(\gamma_1)-f(\gamma_0)\,\d\ppi(\gamma)
\leq{\rm Comp}(\ppi){\rm Lip}(\ppi)
|{\boldsymbol D}f|_{\Pi_{\rm Geo}}^\star(\X)\\
&\leq{\rm Comp}(\eeta){\rm Lip}(\eeta)
|{\boldsymbol D}f|_{\Pi_{\rm Geo}}^\star(\X).
\end{split}\]
This yields \(f\in{\rm BV}^\star(\X)={\rm BV}(\X)\) and
\(|{\boldsymbol D}f|(\X)=|{\boldsymbol D}f|^\star(\X)\leq
|{\boldsymbol D}f|_{\Pi_{\rm Geo}}^\star(\X)\), as desired.
We refer to Theorem \ref{thm:master_G} (and to Remark
\ref{rmk:simpler_K_positive}) for the details.
\end{itemize}
For general \(K\in\R\), the interpolating \(\ppi\in\Pi_{\rm Geo}\)
can be chosen to satisfy just the weaker bound
\begin{equation}\label{eq:intro_BIP_all_K}
{\rm Comp}(\ppi)\leq C\big({\rm Lip}(\eeta)\big){\rm Comp}(\eeta),
\end{equation}
for some function \(C\colon\R^+\to\R^+\), \({\rm L}\mapsto C({\rm L})\)
which converges to \(1\) as \({\rm L}\to 0\) and depends only on
\(K\), \(N\). Therefore, to achieve the analogue of the last passage
above, it is convenient to perform a polygonal interpolation argument:
given any \(\eeta\) and \(n\in\N\), we subdivide \([0,1]\)
into \(n\)-many subintervals of equal length \(1/n\), and for every
\(i=1,\ldots,n\) we denote by \(\eeta_i\) the `restriction' of
\(\eeta\) to the \(i^{th}\) such interval, reparametrised again
on \([0,1]\). Then the \(\infty\)-test plan
\(\ppi_i\in\Pi_{\rm Geo}\) interpolating \(\eeta_i\) satisfies
\({\rm Comp}(\ppi_i)\leq C\big({\rm Lip}(\eeta)/n\big)
{\rm Comp}(\eeta_i)\) and \({\rm Lip}(\ppi_i)\leq{\rm Lip}(\eeta_i)\).
Finally, thanks to a glueing argument we can `patch together' the
\(\ppi_i\)'s, thus obtaining a plan \(\ppi\in\Pi_{\rm Geo}\) which
interpolates \(\eeta\), and satisfies \({\rm Comp}(\ppi)\leq
C\big({\rm Lip}(\eeta)/n\big){\rm Comp}(\eeta)\) and
\({\rm Lip}(\ppi)\leq{\rm Lip}(\eeta)\). Once this is done,
one can proceed as before by using the fact that
\(\lim_n C\big({\rm Lip}(\eeta)/n\big)=1\). 

As it is evident, in the general case $K\in\R$ the polygonal interpolation is a key step to achieve the same conclusion as for $K\geq 0$. We point out that the idea of approximation via a polygonal argument is due to \cite{Lisini07}. However, differently from there, our framework benefits from the synthetic Ricci lower bounds and the \emph{tightness} step, \emph{i.e.}, sending $n$ to infinity, is not a delicate issue in this work thanks to \eqref{eq:intro_BIP_all_K}.

It is worth remarking that the assumption \(\mm(\X)<+\infty\)
in the paper \cite{ACCMCS20} seems to be only of technical nature,
and that it might be possibly removed. As soon as this is done,
the finite assumption on \(\mm\) can be automatically dropped
from all statements in the present paper.
\medskip

Concerning iii): the defining inequality \eqref{eq:intro_BV}
of \({\rm BV}(\X)\) is in integral form, and this may (and
actually does) cause difficulties when one wants to manipulate
\(\infty\)-test plans, for instance taking restrictions. In
other words, if \(\ppi\) is a \(\infty\)-test plan and \(\Gamma\)
is a Borel set of curves with \(0<\ppi(\Gamma)<1\), then one can
consider the test plan \(\tilde\ppi\coloneqq
\ppi(\Gamma)^{-1}\ppi|_\Gamma\). While \({\rm Lip}(\tilde\ppi)
\leq{\rm Lip}(\ppi)\), the best possible bound for the compression
constant is only \({\rm Comp}(\tilde\ppi)\leq\ppi(\Gamma)^{-1}
{\rm Comp}(\ppi)\). This means that the inequality in
\eqref{eq:intro_BV} does not behave well under the restriction
operation.

Therefore, it would be natural to provide a `curvewise' notion
of BV, which allows for more flexibility at the level of the test
plans used in its definition. This goal will be achieved in
Section \ref{s:curvewise_BV}, where we introduce (in Definition
\ref{def:BV_cw}) the space \({\rm BV}^{\sf cw}(\X)\). This concept
is strongly inspired by (and equivalent to) the so-called
\emph{AM-BV space} \({\rm BV}_{AM}(\X)\), quite recently introduced
by O.\ Martio \cite{Martio16,Martio16-2}. Roughly speaking, the main
difference is that in \({\rm BV}_{AM}(\X)\) the defining property
is required to hold along \(AM\)-almost every curve (where
\(AM\) stands for the \emph{\(AM\)-modulus} \cite{Martio16}),
while in \({\rm BV}^{\sf cw}(\X)\) the exceptional curves are
measured by using \(\infty\)-test plans. We will prove in Theorem
\ref{thm:equiv_BV_cw} (see also Corollary \ref{cor:equiv_BV_cw})
that \({\rm BV}^{\sf cw}(\X)\) and \({\rm BV}(\X)\) coincide,
while in Theorem \ref{thm:master_cw_plan} we show (using
Theorem \ref{thm:countable_master_tp}) that it is possible to find
a single master test plan for the curvewise BV space, concentrated
on geodesics when the underlying metric measure space is
non-branching \({\sf CD}(K,N)\). As a side result of independent
interest, we will prove in Appendix \ref{app:AM-BV}, more specifically
in Theorem \ref{thm:equiv_BV_AM}, that the spaces
\({\rm BV}_{AM}(\X)\) and \({\rm BV}(\X)\) coincide
on every metric measure space.
To the best of our knowledge, this kind of equivalence result
has been previously addressed only in the framework of doubling
spaces supporting a Poincar\'{e} inequality, see \cite{DCEBKS19}.
\medskip

Finally, about iv): while the BV-theory on metric
measure spaces is well-established by now, the \(W^{1,1}\)-theory
seems to be much more complex to deal with.
In \cite{Ambrosio-DiMarino14,DiMarinoPhD} several definitions of
\(W^{1,1}(\X)\) are proposed, but it is shown that some of them
are not equivalent. In particular, in the ${\sf CD}$-setting, the available identifications \cite{Ambrosio-Pinamonti-Speight15} are mainly based on Doubling and Poincar\'e conditions rather than on the curvature hypothesis, two properties that are known to hold in the $\sf CD$-setting after \cite{Sturm06I,Sturm06II} and \cite{Rajala12}. However, as proven in \cite{GigliHan14},
the situation significantly improves when considering
\({\sf RCD}(K,N)\) spaces, since in this case the various
notions of \(W^{1,1}(\X)\) are in fact all equivalent and they relate to the space \({\rm BV}(\X)\). The ${\sf RCD}(K,N)\coloneqq {\sf RCD}_2(K,N)$ condition was introduced as a `Riemannian' counterpart of the ${\sf CD}$-theory to single out Finsler geometries (here, ${\sf R}$ stands for Riemannian). It was proposed first in the infinite dimensional setting \cite{AmbrosioGigliSavare11-2} (see also \cite{AGMR15} for the case of $\sigma$-finite reference measure) and later in the finite dimensional setting in \cite{Gigli12}. The ${\sf RCD}(K,N)$ condition is formulated adding to the ${\sf CD}$-condition the so-called infinitesimal Hilbertianity property, namely that the Sobolev space $W^{1,2}(\X)$ is Hilbert, and essentially reflects the fact that the underlying space, at small scales, looks `Hilbertian'. For a complete account on this theory, we refer also to \cite{AmbrosioGigliSavare12,AmbrosioMondinoSavare13-2,EKS15} for equivalent formulations and to \cite{CM16} for the equivalence with \emph{reduced} Riemannian curvature dimension conditions.

In this note, it will be important to recall that ${\sf RCD}(K,N)$ spaces for some $N<\infty$ are non-branching, as proven in \cite{Deng20}. Therefore, by appealing to \cite{GigliHan14}, we can prove in Appendix \ref{app:master_W11} the existence of a master plan for
\(W^{1,1}(\X)\) concentrated on geodesics (for \({\sf RCD}(K,N)\)
spaces whose \(\mm\) is finite).
\medskip

We conclude this Introduction by mentioning that the concept of
a \emph{master test plan} was first proposed and investigated by the
second named author in \cite{Pas20}, where \(p\)-Sobolev
spaces \(W^{1,p}(\X)\) with \(p\in(1,\infty)\) were studied.
The corresponding results to Theorems \ref{thm:master_G} and
\ref{thm:countable_master_tp} for \(p\)-Sobolev spaces will be
obtained by the first named author and N.\ Gigli in
\cite{GigliNobili21}.

Finally, we point out that a geometric consequence of the results of the present paper
will be obtained in the forthcoming work \cite{BPS21}, 
where it will be shown that on \({\sf RCD}(K,N)\) spaces each perimeter measure (or, more generally, the
total variation measure of any BV function) is concentrated on
the \(n\)-regular set, where \(n\leq N\) stands for the essential
dimension of \(\X\). This pushes the `constancy-of-dimension'
result proved in \cite{BS19} up to codimension one.
\medskip

\textbf{Acknowledgements.} The authors would like to thank
Nicola Gigli for having suggested Remark \ref{rmk:MCP}.
The second named author was supported by the Balzan project
led by Luigi Ambrosio. The third named author was supported
by the Deutsche Forschungsgemeinschaft (DFG, German Research
Foundation) - through SPP 2026 Geometry at Infinity.
\section{Preliminaries}
\subsection{Metric spaces and measure theory}
We use the shorthand notation \(\mathcal L_1\) to denote
the restriction of the one-dimensional Lebesgue measure \(\mathcal L^1\)
to the unit interval \([0,1]\), namely,
\[
\mathcal L_1\coloneqq\mathcal L^1|_{[0,1]}.
\]
Let \((\X,\sfd)\) be a complete, separable metric space.
The space \(C([0,1],\X)\) of all continuous curves in \(\X\) is a
complete, separable metric space if endowed with the supremum distance
\[
\sfd_{\rm sup}(\gamma,\sigma)\coloneqq\max_{t\in[0,1]}\sfd(\gamma_t,\sigma_t),
\quad\text{ for every }\gamma,\sigma\in C([0,1],\X).
\]
Given any \(t\in[0,1]\), we define the evaluation map
\(\e_t\colon C([0,1],\X)\to\X\) at time \(t\) as
\(\e_t(\gamma)\coloneqq\gamma_t\) for every \(\gamma\in C([0,1],\X)\).
Moreover, given any \(s,t\in[0,1]\) with \(s\neq t\), we define the
restriction mapping \({\rm restr}_s^t\colon C([0,1],\X)\to C([0,1],\X)\)
as \({\rm restr}_s^t(\gamma)_r\coloneqq\gamma_{(1-r)s+rt}\) for every
\(\gamma\in C([0,1],\X)\) and \(r\in[0,1]\). Observe that the maps
\(\e_t\) and \({\rm restr}_s^t\) are \(1\)-Lipschitz.
\medskip

Given any exponent \(q\in[1,\infty]\), we denote by
\(AC^q([0,1],\X)\) the family of all \(q\)-absolutely continuous
curves in \(\X\), namely, we declare that a given curve
\(\gamma\in C([0,1],\X)\) belongs to the space \(AC^q([0,1],\X)\)
provided there exists a function \(g\in L^q(0,1)\) such that
\begin{equation}\label{eq:def_AC_curve}
\sfd(\gamma_t,\gamma_s)\leq\int_s^t g(r)\,\d r,
\quad\text{ for every }s,t\in[0,1]\text{ with }s\leq t.
\end{equation}
It holds that \(AC^q([0,1],\X)\) is a Borel subset of
\(\big(C([0,1],\X),\sfd_{\rm sup}\big)\). Let us also denote
\[
AC([0,1],\X)\coloneqq AC^1([0,1],\X),\quad
{\rm LIP}([0,1],\X)\coloneqq AC^\infty([0,1],\X).
\]
Observe that \({\rm LIP}([0,1],\X)\subseteq AC^q([0,1],\X)
\subseteq AC([0,1],\X)\) holds for every \(q\in(1,\infty)\).
Given any \(\gamma\in AC^q([0,1],\X)\), we have that
\(|\dot\gamma_t|\coloneqq\lim_{h\to 0}\sfd(\gamma_{t+h},\gamma_t)/|h|\)
exists for \(\mathcal L_1\)-a.e.\ \(t\in[0,1]\).
The resulting \(\mathcal L_1\)-a.e.\ defined function \(|\dot\gamma|\)
belongs to \(L^q(0,1)\) and is the minimal \(g\in L^q(0,1)\)
satisfying \eqref{eq:def_AC_curve}. We refer to
\cite[Theorem 1.1.2]{AmbrosioGigliSavare08} for a proof of these claims.
We define the metric speed functional
\({\sf ms}\colon C([0,1],\X)\times[0,1]\to[0,+\infty]\) as follows:
\[
{\sf ms}(\gamma,t)\coloneqq|\dot\gamma_t|=
\lim_{h\to 0}\frac{\sfd(\gamma_{t+h},\gamma_t)}{|h|},
\quad\text{ whenever }\gamma\in AC([0,1],\X)\text{ and }
\exists\lim_{h\to 0}\frac{\sfd(\gamma_{t+h},\gamma_t)}{|h|},
\]
and \({\sf ms}(\gamma,t)\coloneqq+\infty\) otherwise.
It holds that \(\sf ms\) is a Borel function.
\medskip

We denote by \(\mathscr P(\X)\) the family of all Borel probability
measures on \((\X,\sfd)\). We will always endow \(\mathscr P(\X)\)
with the narrow topology: given any sequence
\((\mu_n)_{n\in\N\cup\{\infty\}}\subseteq\mathscr P(\X)\),
we say that \(\mu_n\) narrowly converges to \(\mu_\infty\)
as \(n\to\infty\) (or, briefly, \(\mu_n\rightharpoonup\mu_\infty\)
as \(n\to\infty\)) provided
\[
\int\varphi\,\d\mu_\infty=\lim_{n\to\infty}\int\varphi\,\d\mu_n,
\quad\text{ for every }\varphi\in C_b(\X),
\]
where \(C_b(\X)\) stands for the space of bounded,
continuous real-valued functions \(\varphi\colon\X\to\R\).
There exists a complete and separable distance \(\sfd_{\mathscr P}\)
which metrises the narrow topology.
A family $ {\mathcal K} \subseteq {\mathscr P}(\X)$ of probabilities is called tight if, for every $\varepsilon>0$, there exists a compact set $K_\varepsilon \subseteq \X$ so that $\mu(K_\varepsilon) \ge 1-\varepsilon$ for every $\mu \in {\mathcal K}$. For later use, we recall that tightness can be characterised thanks to Prokhorov's Theorem as follows: $ {\mathcal K} \subseteq {\mathscr P}(\X)$ is tight if and only if there exists a functional $\psi \colon \X \to [0,+\infty]$ with compact sublevels so that $\sup_{\mu \in {\mathcal K}}\int \psi \, \d \mu <\infty$.
For any \(\mu\in\mathscr P(\X)\), we have that
\(\|f\|_{L^q(\mu)}\leq\|f\|_{L^{q'}(\mu)}\) for every
\(q,q'\in[1,\infty]\) with \(q\leq q'\) and \(f\colon\X\to[0,+\infty)\)
Borel, as a consequence of H\"{o}lder's inequality. Moreover, we have that
\begin{equation}\label{eq:continuity_Lq_norm}
\|f\|_{L^\infty(\mu)}=\lim_{q\to\infty}\|f\|_{L^q(\mu)},
\quad\text{ for every }\mu\in\mathscr P(\X)\text{ and }
f\colon\X\to[0,+\infty)\text{ Borel.}
\end{equation}
Indeed, assuming \(\lim_{q\to\infty}\|f\|_{L^q(\mu)}<+\infty\),
for any \(\lambda\in(0,+\infty)\) with
\(\lambda<\|f\|_{L^\infty(\mu)}\) we have that the Borel set
\(S_\lambda\coloneqq\big\{x\in\X\,:\,f(x)\geq\lambda\big\}\)
has finite \(\mu\)-measure, thus by letting \(q\to\infty\) in
\[
\|f\|_{L^q(\mu)}\geq\bigg(\int_{S_\lambda}f^q\,\d\mu\bigg)^{1/q}\geq
\big(\lambda^q\mu(S_\lambda)\big)^{1/q}=\lambda\mu(S_\lambda)^{1/q}
\]
we deduce that \(\lim_{q\to\infty}\|f\|_{L^q(\mu)}\geq\lambda\).
Finally, by letting \(\lambda\nearrow\|f\|_{L^\infty(\mu)}\),
we obtain \eqref{eq:continuity_Lq_norm}.
\subsection{Test plans}
By a metric measure space we mean a triple \((\X,\sfd,\mm)\), where
\[\begin{split}
(\X,\sfd)&\quad\text{ is a complete, separable metric space},\\
\mm\geq 0&\quad\text{ is a boundedly-finite Borel measure on }(\X,\sfd).
\end{split}\]
Let us recall the key notion of test plan, which was introduced
in \cite{AmbrosioGigliSavare11,AmbrosioGigliSavare11-3}.
\begin{definition}[Test plan]\label{def:test_plan}
Let \((\X,\sfd,\mm)\) be a metric measure space. Let \(q\in[1,\infty]\)
be given. Then a measure \(\ppi\in\mathscr P\big(C([0,1],\X)\big)\) is
said to be a \(q\)-test plan on \((\X,\sfd,\mm)\) provided
\begin{subequations}\begin{align}\label{eq:def_test_plan1}
&\exists\,{\rm C}>0:\quad\forall t\in[0,1],\quad(\e_t)_\#\ppi\leq{\rm C}\mm,\\
\label{eq:def_test_plan2}
&\|{\sf ms}\|_{L^q(\sppi\otimes\mathcal L_1)}<+\infty.
\end{align}\end{subequations}
The compression constant \({\rm Comp}(\ppi)\) of \(\ppi\) is defined
as the minimal \({\rm C}>0\) satisfying \eqref{eq:def_test_plan1}. The
kinetic \(q\)-energy \({\rm Ke}_q(\ppi)\) of \(\ppi\) is defined as
\[
{\rm Ke}_q(\ppi)\coloneqq\left\{\begin{array}{ll}
\|{\sf ms}\|_{L^q(\sppi\otimes\mathcal L_1)}^q,\\
\|{\sf ms}\|_{L^\infty(\sppi\otimes\mathcal L_1)},
\end{array}\quad\begin{array}{ll}
\text{ if }q<\infty,\\
\text{ if }q=\infty.
\end{array}\right.
\]
We denote by \(\Pi_q(\X)\) the family of all \(q\)-test plans
on \((\X,\sfd,\mm)\).
\end{definition}
Observe that a \(q\)-test plan must be concentrated on \(AC^q([0,1],\X)\).
Hereafter, we shall mostly focus on \(\infty\)-test plans. In this regard,
we introduce the alternative notation
\[
{\rm Lip}(\ppi)\coloneqq{\rm Ke}_\infty(\ppi),\quad
\text{ for every }\infty\text{-test plan }\ppi\text{ on }(\X,\sfd,\mm).
\]
This choice of terminology is motivated by the following observation.
\begin{remark}{\rm
Given any \(\infty\)-test plan \(\ppi\) on \((\X,\sfd,\mm)\), the
quantity \({\rm Lip}(\ppi)\) can be equivalently characterised as
the minimal \({\rm L}\geq 0\) such that \(\ppi\) is concentrated on
\({\rm L}\)-Lipschitz curves.

Indeed, by Fubini's theorem we know that \(\ppi\)-a.e.\ \(\gamma\)
is Lipschitz and satisfies \(|\dot\gamma_t|\leq{\rm Lip}(\ppi)\) for
\(\mathcal L_1\)-a.e.\ \(t\in[0,1]\), thus accordingly
\(\sfd(\gamma_t,\gamma_s)\leq\int_s^t|\dot\gamma_r|\,\d r\leq
{\rm Lip}(\ppi)|t-s|\) for every \(s,t\in[0,1]\) with \(s\leq t\),
which shows that \(\gamma\) is \({\rm Lip}(\ppi)\)-Lipschitz.
On the other hand, if \(\ppi\) is concentrated on \({\rm L}\)-Lipschitz
curves for some \({\rm L}\geq 0\), then for \(\ppi\)-a.e.\ curve
\(\gamma\) we have that
\[
|\dot\gamma_t|=\lim_{h\to 0}\frac{\sfd(\gamma_{t+h},\gamma_t)}{|h|}
\leq\lim_{h\to\infty}\frac{{\rm L}|(t+h)-t|}{|h|}={\rm L},\quad
\text{ for }\mathcal L_1\text{-a.e.\ }t\in[0,1],
\]
whence it follows that \({\rm Lip}(\ppi)=
\|{\sf ms}\|_{L^\infty(\sppi\otimes\mathcal L_1)}\leq{\rm L}\).
Therefore, the claim is achieved.
\fr}\end{remark}
\begin{remark}\label{rmk:cont_eval_meas}{\rm
We point out that, given any \(t\in[0,1]\), it holds that
\[
\mathscr P\big(C([0,1],\X)\big)\ni\ppi\mapsto(\e_t)_\#\ppi
\in\mathscr P(\X),\quad\text{ is continuous,}
\]
where both the domain and the target are endowed with the narrow
topology. Indeed, the continuity of \(\e_t\) ensures that
\(\varphi\circ\e_t\in C_b\big(C([0,1],\X)\big)\) whenever
\(\varphi\in C_b(\X)\), so that if we assume
\(\ppi_n\rightharpoonup\ppi\), then we have that
\(\int\varphi\,\d(\e_t)_\#\ppi_n\to\int\varphi\,\d(\e_t)_\#\ppi\)
for every \(\varphi\in C_b(\X)\).

In particular, if we further assume that \(\ppi\) and
\((\ppi_n)_{n\in\N}\) are \(q\)-test plans for some \(q\in[1,\infty]\)
and \({\rm C}\coloneqq\limi_{n\to\infty}{\rm Comp}(\ppi_n)<+\infty\),
then it holds that \({\rm Comp}(\ppi)\leq{\rm C}\). This can be proved
by observing that, chosen a subsequence \((n_i)_{i\in\N}\) which
satisfies \({\rm C}=\lim_{i\to\infty}{\rm Comp}(\ppi_{n_i})\),
we have \(\int\varphi\,\d(\e_t)_\#\ppi=
\lim_i\int\varphi\,\d(\e_t)_\#\ppi_{n_i}\leq{\rm C}\int\varphi\,\d\mm\)
for every \(t\in[0,1]\) and \(\varphi\in C_b(\X)^+\).
\fr}\end{remark}
The kinetic \(q\)-energy can be extended to a functional
\({\rm Ke}_q\colon\mathscr P\big(C([0,1],\X)\big)\to[0,+\infty]\)
by declaring that \({\rm Ke}_q(\ppi)\coloneqq+\infty\) whenever \(\ppi\)
is not a \(q\)-test plan. It is well-known that \({\rm Ke}_q\) is
lower semicontinuous for each \(q\in(1,\infty)\) when its domain
is endowed with the narrow topology. As a consequence,
we obtain the following Mosco-type convergence result.
\begin{proposition}\label{prop:Mosco_conv_tp}
Let \((\X,\sfd,\mm)\) be a metric measure space. Let
\((q_n)_{n\in\N}\subseteq(1,\infty)\) be any sequence satisfying
\(q_n\nearrow\infty\) as \(n\to\infty\). Then the following properties
are verified:
\begin{itemize}
\item[\(\rm i)\)] If \((\ppi_n)_{n\in\N\cup\{\infty\}}\subseteq
\mathscr P\big(C([0,1],\X)\big)\) and \(\ppi_n\rightharpoonup\ppi_\infty\)
as \(n\to\infty\), then
\begin{equation}\label{eq:lsc_Ke}
{\rm Lip}(\ppi_\infty)\leq\limi_{n\to\infty}{\rm Ke}_{q_n}^{1/q_n}(\ppi_n).
\end{equation}
In particular, it holds that \({\rm Lip}\colon
\mathscr P\big(C([0,1],\X)\big)\to[0,+\infty]\) is lower semicontinuous.
\item[\(\rm ii)\)] Given any \(\ppi\in\mathscr P\big(C([0,1],\X)\big)\),
it holds that
\[
{\rm Lip}(\ppi)\geq\lims_{n\to\infty}{\rm Ke}_{q_n}^{1/q_n}(\ppi).
\]
In particular, it holds that
\({\rm Ke}_{q_n}^{1/q_n}(\ppi)\to{\rm Lip}(\ppi)\) as \(n\to\infty\).
\end{itemize}
\end{proposition}
\begin{proof}
By applying \eqref{eq:continuity_Lq_norm} with
\(\mu\coloneqq\ppi\) and \(f\coloneqq{\sf ms}\), we get
\({\rm Lip}(\ppi)=\lim_{n\to\infty}{\rm Ke}_{q_n}^{1/q_n}(\ppi)\)
for every \(\ppi\in\mathscr P\big(C([0,1],\X)\big)\), thus proving ii).
To prove i), fix \((\ppi_n)_{n\in\N\cup\{\infty\}}\subseteq
\mathscr P\big(C([0,1],\X)\big)\) such that
\(\ppi_n\rightharpoonup\ppi_\infty\) as \(n\to\infty\).
Given any \(k,n\in\N\) with \(n\geq k\), we have that \(q_n\geq q_k\) and
thus \({\rm Ke}_{q_k}^{1/q_k}(\ppi_n)\leq{\rm Ke}_{q_n}^{1/q_n}(\ppi_n)\).
Therefore, the lower semicontinuity of each \({\rm Ke}_{q_k}^{1/q_k}\) gives
\[
{\rm Lip}(\ppi_\infty)=
\lim_{k\to\infty}{\rm Ke}_{q_k}^{1/q_k}(\ppi_\infty)\leq
\lim_{k\to\infty}\limi_{n\to\infty}{\rm Ke}_{q_k}^{1/q_k}(\ppi_n)
\leq\limi_{n\to\infty}{\rm Ke}_{q_n}^{1/q_n}(\ppi_n),
\]
which proves i).
\end{proof}

In the next section, we need the compactness result for
\(\infty\)-test plans we are going to present. Before doing so,
we recall a variant of the Arzel\`{a}--Ascoli theorem
(see \cite[Proposition 2.1]{PaolStep12}).
\begin{theorem}[Arzel\`{a}--Ascoli theorem revisited]\label{thm:AA}
Let \((\X,\sfd)\) be a metric space. Let \(\mathcal K\) be a
closed subset of \(\big(C([0,1],\X),\sfd_{\rm sup}\big)\) which
satisfies the following properties:
\begin{itemize}
\item[\(\rm i)\)] \(\mathcal K\subseteq{\rm LIP}([0,1],\X)\) and
\(\sup\big\{{\rm Lip}(\gamma)\,:\,\gamma\in\mathcal K\big\}<+\infty\).
\item[\(\rm ii)\)] Given any \(\varepsilon>0\), there exists a compact
set \(K_\varepsilon\subseteq\X\) such that
\[
\mathcal L^1\big(\big\{t\in[0,1]\,:\,\gamma_t\notin K_\varepsilon
\big\}\big)\leq\varepsilon,\quad\text{ for every }\gamma\in\mathcal K.
\]
\end{itemize}
Then the set \(\mathcal K\) is compact with respect
to the distance \(\sfd_{\rm sup}\).
\end{theorem}
\begin{proposition}[Compactness result for \(\infty\)-test plans]
\label{prop:compactness_tp}
Let \((\X,\sfd,\mm)\) be a metric measure space. Let \((\ppi_n)_{n\in\N}\)
be a sequence of \(\infty\)-test plans on \((\X,\sfd,\mm)\)
such that \(\bigcup_{n\in\N}{\rm spt}\big((\e_0)_\#\ppi_n\big)\)
is bounded, \(\sup_{n\in\N}{\rm Lip}(\ppi_n)<+\infty\),
and \(\sup_{n\in\N}{\rm Comp}(\ppi_n)<+\infty\).
Then there exist a subsequence \((\ppi_{n_i})_{i\in\N}\) of
\((\ppi_n)_{n\in\N}\) and a \(\infty\)-test plan \(\ppi\) on
\((\X,\sfd,\mm)\) such that \(\ppi_{n_i}\rightharpoonup\ppi\)
as \(i\to\infty\).
\end{proposition}
\begin{proof}
Setting \(S\coloneqq\bigcup_{n\in\N}{\rm spt}\big((\e_0)_\#\ppi_n\big)\)
and \({\rm L}\coloneqq\sup_{n\in\N}{\rm Lip}(\ppi_n)\), we have that
the closed \({\rm L}\)-neighbourhood \(B\) of \(S\) is a bounded set
containing \(\bigcup_{n\in\N}{\rm spt}(\ppi_n)\). Indeed,
for every \(n\in\N\) and \(\ppi_n\)-a.e.\ \(\gamma\) we have that
\(\gamma_0\in S\) and \(\sfd(\gamma_t,\gamma_0)\leq{\rm Lip}(\ppi_n)
\leq{\rm L}\) hold for all \(t\in[0,1]\). Recall that \(\mm(B)\) is
finite and thus the measure \(\mm|_B\) is tight. Since
\((\e_t)_\#\ppi_n\leq{\rm C}\mm|_B\) for every \(n\in\N\) and
\(t\in[0,1]\), where we set
\({\rm C}\coloneqq\sup_{n\in\N}{\rm Comp}(\ppi_n)\),
we deduce that the family
\[
\big\{(\e_t)_\#\ppi_n\,:\,n\in\N,\,t\in[0,1]\big\}\subseteq
\mathscr P(\X)\quad\text{ is tight.}
\]
Therefore, there exists a function \(\psi\colon\X\to[0,+\infty]\)
having compact sublevels such that
\begin{equation}\label{eq:compactness_tp_aux}
\sup\bigg\{\int\psi\,\d(\e_t)_\#\ppi_n\,:\,n\in\N,\,t\in[0,1]\bigg\}
<+\infty.
\end{equation}
Now let us define \(\Psi\colon C([0,1],\X)\to[0,+\infty]\) as
\(\Psi(\gamma)\coloneqq\int_0^1\psi(\gamma_t)\,\d t+{\rm Lip}(\gamma)\)
if \(\gamma\in{\rm LIP}([0,1],\X)\) and \(\Psi(\gamma)\coloneqq+\infty\)
otherwise. We claim that \(\Psi\) has compact sublevels. To prove
it amounts to showing that the set \(\mathcal K_\lambda\coloneqq
\big\{\gamma\in C([0,1],\X)\,:\,\Psi(\gamma)\leq\lambda\big\}\)
is compact for any given \(\lambda>0\). First,
\(\mathcal K_\lambda\) is closed: if \((\gamma^n)_{n\in\N}\subseteq
\mathcal K_\lambda\) and \(\gamma\in C([0,1],\X)\) satisfy
\(\lim_{n\to\infty}\sfd_{\rm sup}(\gamma^n,\gamma)=0\), then
\[\begin{split}
\Psi(\gamma)&\leq\int_0^1\limi_{n\to\infty}\psi(\gamma^n_t)\,\d t
+\limi_{n\to\infty}{\rm Lip}(\gamma^n)\leq
\limi_{n\to\infty}\int_0^1\psi(\gamma^n_t)\,\d t+
\limi_{n\to\infty}{\rm Lip}(\gamma^n)\\
&\leq\limi_{n\to\infty}\Psi(\gamma^n)\leq\lambda,
\end{split}\]
where we used the lower semicontinuity of \(\psi\) and
Fatou's lemma. This shows that \(\gamma\in\mathcal K_\lambda\)
and thus \(\mathcal K_\lambda\) is closed. In order to prove
that \(\mathcal K_\lambda\) is compact, we want to use Theorem
\ref{thm:AA}. The set \(\mathcal K_\lambda\) verifies item i) of
Theorem \ref{thm:AA}, since \({\rm Lip}(\gamma)\leq\lambda\)
for all \(\gamma\in\mathcal K_\lambda\). About item ii), fix
\(\varepsilon>0\) and pick the compact set
\(K_{\lambda,\varepsilon}\coloneqq\big\{x\in\X\,:
\,\psi(x)\leq\lambda/\varepsilon\big\}\).
For any \(\gamma\in\mathcal K_\lambda\), one has
\[
\mathcal L_1\big(\big\{t\in[0,1]\,:\,\gamma_t\notin
K_{\lambda,\varepsilon}\big\}\big)=
\mathcal L_1\big(\big\{t\in[0,1]\,:\,\psi(\gamma_t)>
\lambda/\varepsilon\big\}\big)\leq
\frac{\varepsilon}{\lambda}\int_0^1\psi(\gamma_t)\,\d t\leq\varepsilon,
\]
where we used Chebyshev's inequality. This shows that
\(\mathcal K_\lambda\) verifies item ii) of Theorem \ref{thm:AA}
and thus it is compact. All in all, we proved that \(\Psi\) has
compact sublevels. Finally, note that
\[\begin{split}
\sup_{n\in\N}\int\Psi\,\d\ppi_n&\leq
\sup_{n\in\N}\int\!\!\!\int_0^1\psi(\gamma_t)\,\d t\,\d\ppi_n(\gamma)
+\sup_{n\in\N}\int{\rm Lip}(\gamma)\,\d\ppi_n(\gamma)\\
&\leq\sup_{n\in\N}\int_0^1\!\!\!\int\psi\,\d(\e_t)_\#\ppi_n\,\d t
+\sup_{n\in\N}{\rm Lip}(\ppi_n)\\
&\leq\underset{\substack{n\in\N,\\t\in[0,1]}}\sup
\int\psi\,\d(\e_t)_\#\ppi_n+{\rm L}
\overset{\eqref{eq:compactness_tp_aux}}<+\infty.
\end{split}\]
Therefore, an application of Prokhorov's theorem yields the statement.
\end{proof}
\subsection{Functions of bounded \texorpdfstring{\(\Pi\)}{Pi}-variation}
We propose a notion of function having bounded \(\Pi\)-variation
on a metric measure space, where \(\Pi\) is any given family of
\(\infty\)-test plans. It is the natural generalisation of the
concepts introduced in \cite[Section 5.3]{Ambrosio-DiMarino14}
and \cite[Section 4.4.3]{DiMarinoPhD}.
\begin{definition}[Function of bounded \(\Pi\)-variation]\label{def:BV}
Let \((\X,\sfd,\mm)\) be a metric measure space. Let \(\Pi\neq\emptyset\)
be a family of \(\infty\)-test plans on \((\X,\sfd,\mm)\).
Then we declare that \(f\in L^1(\mm)\) belongs to the space
\({\rm BV}_\Pi(\X)\) provided the following properties
are verified:
\begin{itemize}
\item[\(\rm i)\)] Given any \(\ppi\in\Pi\), it holds that
\(f\circ\gamma\in{\rm BV}(0,1)\) for \(\ppi\)-a.e.\ \(\gamma\).
\item[\(\rm ii)\)] There exists a finite Borel measure \(\mu\geq 0\)
on \((\X,\sfd)\) such that for any \(\ppi\in\Pi\) it holds
\begin{equation}\label{eq:def_BV}
\int\gamma_\#|D(f\circ\gamma)|(B)\,\d\ppi(\gamma)\leq
{\rm Comp}(\ppi){\rm Lip}(\ppi)\mu(B),\quad\text{ for every }
B\subseteq\X\text{ Borel.}
\end{equation}
\end{itemize}
We denote by \(|{\boldsymbol D}f|_\Pi\) the least measure
\(\mu\geq 0\) satisfying \eqref{eq:def_BV} for every \(\ppi\in\Pi\).
\end{definition}
\begin{remark}\label{rem:BVwellposed}
{\rm
Let us briefly comment on the well-posedness of Definition
\ref{def:BV}. The bounded compression assumption on \(\ppi\)
ensures that the function \(f\circ\gamma\in L^1(0,1)\) is
\(\ppi\)-a.e.\ independent of the
chosen representative of \(f\), so that item i) is well-defined.
Concerning ii), it is easy to show that the family \(\mathcal D\)
of all Borel sets \(B\subseteq\X\) for which
\(\gamma\mapsto\gamma_\#|D(f\circ\gamma)|(B)\) is Borel measurable
is a Dynkin system. By exploiting the very definition of the total
variation measure \(|D(f\circ\gamma)|\), one can show that the topology
of \((\X,\sfd)\) is contained in \(\mathcal D\), thus an application
of the Dynkin \(\pi\)-\(\lambda\) theorem ensures that \(\mathcal D\)
coincides with the Borel \(\sigma\)-algebra of \((\X,\sfd)\).
Accordingly, the integral appearing in the left-hand side of
\eqref{eq:def_BV} is meaningful for every \(B\subseteq\X\) Borel.
\fr}\end{remark}

When considering the totality of \(\infty\)-test plans
on \((\X,\sfd,\mm)\), we use the shorthand notation
\[
{\rm BV}(\X)\coloneqq{\rm BV}_{\Pi_\infty(\X)}(\X),
\qquad|{\boldsymbol D}f|\coloneqq|{\boldsymbol D}f|_{\Pi_\infty(\X)}
\quad\text{for every }f\in{\rm BV}(\X).\]
The equivalence between \({\rm BV}(\X)\) and the space
\(w-BV(\X,\sfd,\mm)\) defined in \cite[Section 5.3]{Ambrosio-DiMarino14}
(see also \cite[Section 4.4.3]{DiMarinoPhD}) follows from the ensuing result.
\begin{lemma}\label{lem:bdry_cond_BV}
Let \((\X,\sfd,\mm)\) be a metric measure space. Let \(\ppi\) be a
\(\infty\)-test plan on \((\X,\sfd,\mm)\). Suppose \(f\in L^1(\mm)\)
satisfies \(f\circ\gamma\in{\rm BV}(0,1)\) for \(\ppi\)-a.e.\ \(\gamma\).
Fix any \(\bar t\in[0,1]\). Then it holds
\[
|D(f\circ\gamma)|(\{\bar t\})=0,\quad\text{ for }\ppi\text{-a.e.\ }\gamma.
\]
In particular, it holds that
\(\big|f(\gamma_1)-f(\gamma_0)\big|\leq|D(f\circ\gamma)|((0,1))\)
for \(\ppi\)-a.e.\ \(\gamma\).
\end{lemma}
\begin{proof}
Let us just prove the statement in the case where \(\bar t=0\),
since the other \(\bar t\)'s can be treated in a similar way.
To prove it amounts to showing that
\begin{equation}\label{eq:bdry_cond_claim}
\lim_{t\searrow 0}\fint_0^t\big|f(\gamma_s)-f(\gamma_0)\big|\,\d s=0,
\quad\text{ for }\ppi\text{-a.e.\ }\gamma.
\end{equation}
 We know
\cite[Theorems 3.27 and 3.28]{AmbrosioFuscoPallara} that there exists
\(\big\{\lambda_\gamma\,:\,\gamma\in C([0,1],\X)\big\}\subseteq\R\)
such that
\begin{equation}\label{eq:bdry_cond_aux}
\lim_{t\searrow 0}\fint_0^t\big|f(\gamma_s)-\lambda_\gamma\big|\,\d s=0,
\quad\text{ for }\ppi\text{-a.e.\ }\gamma.
\end{equation}
Now fix any sequence \(t_n\searrow 0\). Given any \(\varepsilon>0\),
we can find a function \(f_\varepsilon\in{\rm LIP}_{bs}(\X)\) such that
\(\|f-f_\varepsilon\|_{L^1(\mm)}\leq\varepsilon\). Then for any
\(n\in\N\) we may estimate
\[\begin{split}
&\int\!\!\!\fint_0^{t_n}\big|f(\gamma_s)-f(\gamma_0)\big|\,\d s\,\d\ppi(\gamma)\\
\leq\,&\int\!\!\!\fint_0^{t_n}|f-f_\varepsilon|\circ\e_s\,\d s\,\d\ppi
+\int\!\!\!\fint_0^{t_n}\big|f_\varepsilon(\gamma_s)-
f_\varepsilon(\gamma_0)\big|\,\d s\,\d\ppi(\gamma)
+\int|f_\varepsilon-f|\circ\e_0\,\d\ppi\\
\leq\,&2\,{\rm Comp}(\ppi)\|f-f_\varepsilon\|_{L^1(\mm)}+
{\rm Lip}(f_\varepsilon)\int\!\!\!\fint_0^{t_n}\sfd(\gamma_s,\gamma_0)
\,\d s\,\d\ppi(\gamma)\\
\leq\,&2\,{\rm Comp}(\ppi)\varepsilon+{\rm Lip}(f_\varepsilon){\rm Lip}(\ppi)t_n.
\end{split}\]
By first letting \(n\to\infty\) and then \(\varepsilon\searrow 0\),
we see that \(\lim_{n\to\infty}\int\!\!\!\fint_0^{t_n}
\big|f(\gamma_s)-f(\gamma_0)\big|\,\d s\,\d\ppi(\gamma)=0\).
Hence, up to taking a not relabelled subsequence, we have
\(\lim_{n\to\infty}\fint_0^{t_n}\big|f(\gamma_s)-f(\gamma_0)\big|\,\d s=0\)
for \(\ppi\)-a.e.\ \(\gamma\). Recalling \eqref{eq:bdry_cond_aux},
we conclude that \(\lambda_\gamma=f(\gamma_0)\) and
\(\lim_{t\searrow 0}\fint_0^t\big|f(\gamma_s)-f(\gamma_0)\big|\,\d s=0\)
hold for \(\ppi\)-a.e.\ \(\gamma\). This proves the validity of
the \(\ppi\)-a.e.\ identity in \eqref{eq:bdry_cond_claim}, as desired.
\end{proof}
Before passing to the last result of this section, let us
introduce some further terminology. Given an open set \(\Omega\)
in a metric space \((\X,\sfd)\), we denote by \(\LIP_{loc}(\Omega)\)
the family of all real-valued, locally Lipschitz functions defined
on \(\Omega\), \emph{i.e.}, of all functions \(f\colon\Omega\to\R\)
having the property that for any \(x\in\Omega\) there exists a radius
\(r_x>0\) with \(B_{r_x}(x)\subseteq\Omega\) such that
\(f|_{B_{r_x}(x)}\) is Lipschitz. Given any \(f\in\LIP_{loc}(\Omega)\),
we define the function \(\lip_a(f)\colon\Omega\to[0,+\infty)\) as
\[
\lip_a(f)(x)\coloneqq\lims_{y,z\to x}
\frac{\big|f(y)-f(z)\big|}{\sfd(y,z)},
\quad\text{ if }x\in\Omega\text{ is an accumulation point},
\]
and \(\lip_a(f)(x)\coloneqq 0\) otherwise. We call
\(\lip_a(f)\) the asymptotic Lipschitz constant of \(f\).
\medskip

The space \({\rm BV}(\X)\) admits the following equivalent
characterisation, cf.\ \cite[Theorem 4.5.3]{DiMarinoPhD}.
\begin{theorem}\label{thm:equiv_BV}
Let \((\X,\sfd,\mm)\) be a metric measure space. Then a given function
\(f\in L^1(\mm)\) belongs to \({\rm BV}(\X)\) if and only if there
exists a constant \(C\geq 0\) such that
\begin{equation}\label{eq:def_BV_equiv}
\int f(\gamma_1)-f(\gamma_0)\,\d\ppi(\gamma)\leq{\rm Comp}(\ppi)
{\rm Lip}(\ppi)C,\quad\text{ for every }\ppi\in\Pi_\infty(\X).
\end{equation}
In this case, it holds that \(|{\boldsymbol D}f|(\X)\) coincides with
the minimal constant \(C\geq 0\) satisfying \eqref{eq:def_BV_equiv}.

Moreover, given any function \(f\in{\rm BV}(\X)\) and an open set
\(\Omega\subseteq\X\), there exists a sequence \((f_n)_{n\in\N}
\subseteq{\rm LIP}_{loc}(\Omega)\cap L^1(\mm|_\Omega)\) such that
\(f_n\to f\) in \(L^1(\mm|_\Omega)\) and \(\int_\Omega{\rm lip}_a(f_n)
\,\d\mm\to|{\boldsymbol D}f|(\Omega)\).
\end{theorem}
\subsection{Optimal transport}
In this subsection we recall some basic notions in Optimal Transport
theory. The problem of optimal transportation dates back to G.\ Monge \cite{Monge81} and seeks for optimal transport maps minimising the transportation cost of two probability measures in the Euclidean space. Here instead we consider the more general formulation of
L.\ Kantorovich \cite{Kantorovich40} in terms of optimal plans and follow modern approaches to the theory on Polish metric spaces,
referring for instance to \cite{Villani09,AmbrosioGigli11} for a thorough presentation of this topic.
\medskip

Given a complete and separable metric space \((\X,\sfd)\) and
an exponent \(q\in[1,\infty]\), we denote by \(\mathscr P_q(\X)\)
the space of all Borel probability measures \(\mu\in\mathscr P(\X)\)
having finite \(q^{th}\)-moment, \emph{i.e.}, satisfying
\(\sfd(\cdot,\bar x)\in L^q(\mu)\) for some (and, thus, for any)
\(\bar x\in\X\). Observe that \(\mathscr P_{q'}(\X)\subseteq
\mathscr P_q(\X)\) for every \(q,q'\in[1,\infty]\) with \(q\leq q'\).
Given any \(\mu_0,\mu_1\in\mathscr P(\X)\), we say
that a probability measure \(\alpha\in\mathscr P(\X\times\X)\)
is an admissible plan between \(\mu_0\) and \(\mu_1\) provided
\((P_0)_\#\alpha=\mu_0\) and \((P_1)_\#\alpha=\mu_1\), where
\(P_0,P_1\colon\X\times\X\to\X\) stand for the projection maps
\(P_0(x,y)\coloneqq x\) and \(P_1(x,y)\coloneqq y\). The family
of all admissible plans between \(\mu_0\) and \(\mu_1\) is
denoted by \({\rm Adm}(\mu_0,\mu_1)\). For \(q<\infty\),
the \(q\)-Wasserstein distance \(W_q\) on \(\mathscr P_q(\X)\)
is defined as
\begin{equation}\label{eq:def_Wq}
W_q^q(\mu_0,\mu_1)\coloneqq\inf_{\alpha\in{\rm Adm}(\mu_0,\mu_1)}
\int\sfd^q(x,y)\,\d\alpha(x,y)
,\quad\text{ for every }
\mu_0,\mu_1\in\mathscr P_q(\X),
\end{equation}
while in the limit case \(q=\infty\), the \(\infty\)-Wasserstein distance
\(W_\infty\) on \(\mathscr P_\infty(\X)\) is defined as
\begin{equation}\label{eq:def_Winfty}
W_\infty(\mu_0,\mu_1)\coloneqq\inf_{\alpha\in{\rm Adm}(\mu_0,\mu_1)}
\underset{\alpha\text{-a.e.\ }(x,y)}{\rm ess\,sup}\sfd(x,y),\quad
\text{ for every }\mu_0,\mu_1\in\mathscr P_\infty(\X).
\end{equation}
It is well-known (see, \emph{e.g.}, \cite{GivSho84,ChaDeP08}) that $W_\infty$ is the monotone limit of the $q$-Wasserstein distances:
\begin{equation}\label{eq:lim_Wq}
W_\infty(\mu_0,\mu_1)=\lim_{q\to\infty}W_q(\mu_0,\mu_1),
\quad\text{ for every }\mu_0,\mu_1\in\mathscr P_\infty(\X).
\end{equation}
We denote by \({\rm Opt}_q(\mu_0,\mu_1)\) the set of all optimal
plans between \(\mu_0\) and \(\mu_1\), \emph{i.e.}, of all minimisers
of \eqref{eq:def_Wq} or \eqref{eq:def_Winfty}.
By using the direct method in the Calculus of Variations, one can
readily show that \({\rm Opt}_q(\mu_0,\mu_1)\neq\emptyset\) for
every \(q\in[1,\infty]\) and \(\mu_0,\mu_1\in\mathscr P_q(\X)\).
We call $(W_q,\mathscr P_q(\X))$ the $q$-Wasserstein space and
recall that it is complete if $(\X,\sfd)$ is a complete metric space.
\medskip

In the sequel, we shall also need the dynamical formulation of the
Optimal Transport problem.
Recall that, for our main proof, we need to perform several (polygonal) interpolations between probability measures via Wasserstein geodesics. This dynamical viewpoint was first discovered by R.\ McCann \cite{McCann97}, where he derived the notion of \emph{displacement interpolations} in the Euclidean space. Here instead, we shall work on the more general setting of Polish geodesic spaces and consider \emph{optimal dynamical plans} \cite[Theorem 2.10]{AmbrosioGigli11} (whose proof is inspired by \cite{Lisini07} and extends the previous works \cite{Lott-Villani09,Villani09} for compact and locally compact spaces).
Before passing to its description, we
have to fix some further terminology. By a geodesic in \((\X,\sfd)\)
we mean a curve \(\gamma\in C([0,1],\X)\) satisfying
\(\sfd(\gamma_t,\gamma_s)=|t-s|\,\sfd(\gamma_0,\gamma_1)\)
for every \(t,s\in[0,1]\). Thanks to the triangle inequality,
this condition is equivalent to requiring that
\(\sfd(\gamma_t,\gamma_s)\leq|t-s|\,\sfd(\gamma_0,\gamma_1)\) for every
\(t,s\in[0,1]\). We denote by \({\rm Geo}(\X)\) the space of all
geodesics in \(C([0,1],\X)\). It is easy to show that \({\rm Geo}(\X)\)
is a closed subset of \(\big(C([0,1],\X),\sfd_{\rm sup}\big)\).

Given an exponent \(q\in[1,\infty]\) and measures
\(\mu_0,\mu_1\in\mathscr P_q(\X)\), we define
\[
{\rm OptGeo}_q(\mu_0,\mu_1)\coloneqq\Big\{
\ppi\in\mathscr P\big({\rm Geo}(\X)\big)\,:\,
(\e_i)_\#\ppi=\mu_i\;\forall i=0,1,\,
(\e_0,\e_1)_\#\ppi\in{\rm Opt}_q(\mu_0,\mu_1)\Big\}.
\]
The elements of \({\rm OptGeo}_q(\mu_0,\mu_1)\) are called
\(q\)-optimal dynamical plans between \(\mu_0\) and \(\mu_1\).
\begin{remark}{\rm
We point out that if \(q\in[1,\infty]\), \(\mu_0,\mu_1\in
\mathscr P_q(\X)\), and \(\ppi\in{\rm OptGeo}_q(\mu_0,\mu_1)\), then
it holds that \(t\mapsto(\e_t)_\#\ppi\) is a geodesic curve in the
metric space \(\big(\mathscr P_q(\X),W_q\big)\) and that
\begin{subequations}\begin{align}\label{eq:dynOPq}
{\rm Ke}_q(\ppi)=W_q^q(\mu_0,\mu_1),&\quad\text{ if }q<\infty,\\
\label{eq:dynOPinfty}
{\rm Lip}(\ppi)=W_\infty(\mu_0,\mu_1),&\quad\text{ if }q=\infty.
\end{align}\end{subequations}
The case \(q<\infty\) is well-known. Concerning the case \(q=\infty\):
for \(s,t\in[0,1]\) with \(s\leq t\) it holds
\[
W_\infty\big((\e_s)_\#\ppi,(\e_t)_\#\ppi\big)
\leq\underset{\sppi\text{-a.e.\ }\gamma}{\rm ess\,sup}
\,\sfd(\gamma_s,\gamma_t)=(t-s)\,\underset{\sppi\text{-a.e.\ }\gamma}
{\rm ess\,sup}\,\sfd(\gamma_0,\gamma_1)=(t-s)W_\infty(\mu_0,\mu_1),
\]
where in the first inequality we used the fact that
\((\e_s,\e_t)_\#\ppi\in{\rm Adm}\big((\e_s)_\#\ppi,(\e_t)_\#\ppi\big)\).
It follows that \([0,1]\ni t\mapsto(\e_t)_\#\ppi
\in\mathscr P_\infty(\X)\) is a \(W_\infty\)-geodesic.
By plugging \(s=0\) and \(t=1\) in the previous estimates,
we obtain the identity in \eqref{eq:dynOPinfty}.
\fr}\end{remark}
\begin{proposition}\label{prop:existence_optgeo_infty}
Let \((\X,\sfd,\mm)\) be a metric measure space.
Let \(q_n\nearrow\infty\) be a given sequence.
Fix any \(\mu_0,\mu_1\in\mathscr P_\infty(\X)\) and suppose there
exists a sequence \((\ppi_n)_{n\in\N}\subseteq\Pi_\infty(\X)\) of
\(\infty\)-test plans \(\ppi_n\in{\rm OptGeo}_{q_n}(\mu_0,\mu_1)\) such
that \(\sup_{n\in\N}{\rm Comp}(\ppi_n)<+\infty\).
Then there exists a \(\infty\)-test plan
\(\ppi\in{\rm OptGeo}_\infty(\mu_0,\mu_1)\) such that
\(\ppi_{n_i}\rightharpoonup\ppi\) as \(i\to\infty\)
for some subsequence \((\ppi_{n_i})_{i\in\N}\).
\end{proposition}
\begin{proof}
First, notice that \({\rm spt}(\mu_0)=\bigcup_{n\in\N}
{\rm spt}\big((\e_0)_\#\ppi_n\big)\) is bounded. Moreover,
calling \(D\) the diameter of \({\rm spt}(\mu_0)\cup{\rm spt}(\mu_1)\),
we claim that \(\sup_{n\in\N}{\rm Lip}(\ppi_n)\leq D\). Indeed, given
any \(n\in\N\), we know that \(\ppi_n\)-a.e.\ curve \(\gamma\) is a
geodesic satisfying \(\gamma_0\in{\rm spt}(\mu_0)\) and
\(\gamma_1\in{\rm spt}(\mu_1)\), so that accordingly
\({\rm Lip}(\gamma)=\sfd(\gamma_0,\gamma_1)\leq D\). Therefore,
since \(\sup_{n\in\N}{\rm Comp}(\ppi_n)<+\infty\) by assumption,
we are in a position to apply Proposition \ref{prop:compactness_tp},
thus obtaining that there exists a \(\infty\)-test plan \(\ppi\)
on \((\X,\sfd,\mm)\) such that \(\ppi_n\rightharpoonup\ppi\),
up to a not relabelled subsequence in \(n\). By applying Remark
\ref{rmk:cont_eval_meas}, we see that \((\e_0)_\#\ppi=\mu_0\)
and \((\e_1)_\#\ppi=\mu_1\).
Being \({\rm Geo}(\X)\) a closed subset of \(C([0,1],\X)\),
one has \(\ppi({\rm Geo}(\X)^c)\leq\limi_n\ppi_n({\rm Geo}(\X)^c)=0\),
which shows that \(\ppi\) is concentrated on \({\rm Geo}(\X)\).
In order to prove that \(\ppi\in{\rm OptGeo}_\infty(\mu_0,\mu_1)\),
it only remains to show that
\(\alpha\coloneqq(\e_0,\e_1)_\#\ppi\in{\rm Opt}_\infty(\mu_0,\mu_1)\).
This readily follows from the estimates
\[\begin{split}
\underset{\alpha\text{-a.e.\ }(x,y)}{\rm ess\,sup}\sfd(x,y)
&\overset{\phantom{\eqref{eq:dynOPq}}}=\underset
{\sppi\text{-a.e.\ }\gamma}{\rm ess\,sup}\,\sfd(\gamma_0,\gamma_1)
={\rm Lip}(\ppi)\overset{\eqref{eq:lsc_Ke}}\leq
\lim_{n\to\infty}{\rm Ke}_{q_n}^{1/q_n}(\ppi_n)\\
&\overset{\eqref{eq:dynOPq}}=\lim_{n\to\infty}W_{q_n}(\mu_0,\mu_1)
\overset{\eqref{eq:lim_Wq}}=W_\infty(\mu_0,\mu_1).
\end{split}\]
Therefore, the statement is achieved.
\end{proof}
\subsection{The Curvature-Dimension condition}
Let us recall the basic theory of lower Ricci curvature bounds
for metric measure spaces which was introduced by Lott--Villani
\cite{Lott-Villani09} and Sturm \cite{Sturm06I,Sturm06II}, the
so-called Curvature-Dimension condition \({\sf CD}(K,N)\), where
\(K\) and \(N\) stand for the lower Ricci curvature bound and the
upper dimension bound, respectively.
\medskip

To begin with, we recall the definition of the volume distortion
coefficients \(\tau_{K,N}^{(t)}\). Given any \(K\in\R\),
\(N\in(1,\infty)\), \(t\in[0,1]\), and \(\theta\in[0,+\infty)\),
we define
\[\begin{split}
\tau_{K,N}^{(t)}(\theta)&\coloneqq\left\{\begin{array}{ll}
t^{1/N}\bigg(\frac{\sin\big(t\theta\sqrt{K/(N-1)}\big)}
{\sin\big(\theta\sqrt{K/(N-1)}\big)}\bigg)^{(N-1)/N},
\quad&\text{ if }K>0\text{ and }K\theta^2\leq(N-1)\pi^2,\\
+\infty,\quad&\text{ if }K>0\text{ and }K\theta^2>(N-1)\pi^2,\\
t,\quad&\text{ if }K=0,\\
t^{1/N}\bigg(\frac{\sinh\big(t\theta\sqrt{-K/(N-1)}\big)}
{\sinh\big(\theta\sqrt{-K/(N-1)}\big)}\bigg)^{(N-1)/N},
\quad&\text{ if }K>0.
\end{array}\right.
\end{split}\]

Given a metric measure space \((\X,\sfd,\mm)\) and any coefficient
\(N\in(1,\infty)\), following \cite{Sturm06II} we define the
\(N\)-R\'{e}nyi relative entropy functional
\(\mathcal E_N\colon\mathscr P(\X)\to[0,+\infty]\) as
\[
\mathcal E_N(\mu)\coloneqq\int\rho^{1-\frac{1}{N}}\,\d\mm,
\quad\text{ for every }\mu\in\mathscr P(\X),\;\mu=\rho\mm+\mu^s
\text{ with }\mu^s\perp\mm.
\]
We are now in a position to enunciate the definition of a
\({\sf CD}_q(K,N)\) space.
\begin{definition}[\({\sf CD}_q(K,N)\) space]\label{def:CDq}
Let \((\X,\sfd,\mm)\) be a metric measure space. Let \(q\in(1,\infty)\)
be given. Then \((\X,\sfd,\mm)\) is said to be a \({\sf CD}_q(K,N)\)
space, for some \(K\in\R\) and \(N\in(1,\infty)\), provided the
following property holds: given any \(\mu_0,\mu_1\in\mathscr P_q(\X)\),
there exists a \(q\)-optimal dynamical plan
\(\ppi\in{\rm OptGeo}_q(\mu_0,\mu_1)\) such that
\(\mu_t\coloneqq(\e_t)_\#\ppi\ll\mm\) for every \(t\in[0,1]\) and
\[
\mathcal E_{N'}(\mu_t)\geq\int\rho_0(\gamma_0)^{-\frac{1}{N'}}
\tau_{K,N'}^{(1-t)}\big(\sfd(\gamma_0,\gamma_1)\big)+
\rho_1(\gamma_1)^{-\frac{1}{N'}}
\tau_{K,N'}^{(t)}\big(\sfd(\gamma_0,\gamma_1)\big)\,\d\ppi(\gamma),
\]
for every \(N'\geq N\) and \(t\in[0,1]\), where we denote
\(\mu_0=\rho_0\mm\) and \(\mu_1=\rho_1\mm\). When \(q=2\),
we just write \({\sf CD}(K,N)\) in place of \({\sf CD}_2(K,N)\).
\end{definition}
Recall that a metric space \((\X,\sfd)\) is said to be non-branching
provided the implication
\begin{equation}
(\gamma_0,\gamma_t)=(\sigma_0,\sigma_t),\,\text{ for some }t\in(0,1)
\quad\Longrightarrow\quad\gamma=\sigma,
\label{eq:nonbranching}
\end{equation}
holds for every \(\gamma,\sigma\in{\rm Geo}(\X)\).
In the non-branching \({\sf CD}_q(K,N)\) setting the following property
holds: if \(\mu_0,\mu_1\in\mathscr P_q(\X)\) and \(\mu_0\ll\mm\),
then \({\rm OptGeo}_q(\mu_0,\mu_1)\) contains a unique element (see, \emph{e.g.}, \cite[Proposition 2.6]{AmbrosioGigli11} for $q=2$, but the proof goes along the same line for general $q \in (1,\infty)$).
The following result can be obtained by arguing as in
\cite[Proposition 4.2 iv)]{Sturm06II} or as in the proof
of \cite[Theorem 30.32]{Villani09}.
\begin{proposition}\label{prop:bounds_for_OptGeo}
Let \((\X,\sfd,\mm)\) be a non-branching \({\sf CD}_q(K,N)\) space,
for some \(K\in\R\), \(N\in(1,\infty)\), and \(q\in(1,\infty)\). Let
\(\mu_0,\mu_1\in\mathscr P_q(\X)\) be such that \(\mu_0,\mu_1\ll\mm\).
Denote \(\mu_0=\rho_0\mm\) and \(\mu_1=\rho_1\mm\). Let \(\ppi\)
be the unique element of \({\rm OptGeo}_q(\mu_0,\mu_1)\). Then
\(\mu_t\coloneqq(\e_t)_\#\ppi\ll\mm\) for every \(t\in[0,1]\).
Moreover, for every \(t\in[0,1]\) it holds that
\begin{equation}
\rho_t(\gamma_t)^{-\frac{1}{N}}\geq\rho_0(\gamma_0)^{-\frac{1}{N}}
\tau_{K,N}^{(1-t)}\big(\sfd(\gamma_0,\gamma_1)\big)+
\rho_1(\gamma_1)^{-\frac{1}{N}}
\tau_{K,N}^{(t)}\big(\sfd(\gamma_0,\gamma_1)\big),
\quad\text{ for }\ppi\text{-a.e.\ }\gamma,
\end{equation}
where we set \(\mu_t=\rho_t\mm\).
\end{proposition}
We also recall the following important result, which was proved
in \cite[Theorem 1.1]{ACCMCS20}. It states that
on non-branching spaces (whose reference measure is finite), the
\({\sf CD}_q(K,N)\) condition is in fact independent of \(q\).
\begin{theorem}[Equivalence of \({\sf CD}_q\) on \(q>1\)]
\label{thm:equiv_CD_q}
Let \((\X,\sfd,\mm)\) be a non-branching \({\sf CD}(K,N)\) space,
for some \(K\in\R\) and \(N\in(1,\infty)\). Suppose that the measure
\(\mm\) is finite. Then \((\X,\sfd,\mm)\) is a \({\sf CD}_q(K,N)\)
space for every \(q\in(1,\infty)\).
\end{theorem}
The above results in fact hold under a weaker assumption,
called \(q\)-essential non-branching; see \cite{RajalaSturm12}
for the definition of such condition.
\section{Existence of master families of test plans}
In this section, we are concerned with the existence problem
of suitable \emph{master families} for the space of functions
having bounded variation. Roughly, by a master family
for BV we mean a family of \(\infty\)-test plans
on a given metric measure space which is capable of detecting
all BV functions and their total variation measures.
The precise definition is the following:
\begin{definition}[Master family for \(\rm BV\)]
Let \((\X,\sfd,\mm)\) be a metric measure space. Then a given
family \(\Pi\) of \(\infty\)-test plans on \((\X,\sfd,\mm)\)
is said to be a master family for \({\rm BV}(\X)\) provided
\[
{\rm BV}_\Pi(\X)={\rm BV}(\X),\quad
|{\boldsymbol D}f|_\Pi=|{\boldsymbol D}f|\,\text{ for every }
f\in{\rm BV}(\X).
\]
\end{definition}
It is worth pointing out that that the totality \(\Pi_\infty(\X)\)
of \(\infty\)-test plans on \((\X,\sfd,\mm)\) is, by definition,
a master family for \({\rm BV}(\X)\). However, it is not clear
\emph{a priori} whether any strictly smaller family of \(\infty\)-test
plans can be a master family for BV.
\medskip

The aim of this section is twofold: to show that in the setting
on non-branching \({\sf CD}(K,N)\) spaces, the set of \(\infty\)-test
plans concentrated on geodesics is a master family for BV; to prove
that on an arbitrary metric measure space, it is possible to find a
countable master family for BV. We will achieve these goals in
Subsections \ref{ss:master_geod_CD} and \ref{ss:master_seq},
respectively.
\subsection{Master geodesic plans on non-branching
\texorpdfstring{\(\sf CD\)}{CD} spaces}\label{ss:master_geod_CD}
Here we prove that on non-branching \({\sf CD}(K,N)\) spaces
(with finite reference measure) the \(\infty\)-test plans concentrated
on geodesics form a master family for BV; in fact, we prove a stronger
result, namely that those \(\infty\)-test plans which are
\(\infty\)-optimal dynamical plans between their marginals
form a master family for BV. As described in the Introduction,
the first step is to prove existence of interpolating
\(\infty\)-optimal geodesic plans having `well-behaved'
compression constants.
\begin{theorem}[Existence of `good' \(\infty\)-optimal dynamical plans]
\label{thm:CD_are_BIP}
Let \((\X,\sfd,\mm)\) be a non-branching \({\sf CD}(K,N)\) space,
with \(K\in\R\) and \(N<\infty\). Suppose the measure \(\mm\)
is finite. Then there exists a nondecreasing function
\(C\colon(0,+\infty)\to(0,+\infty)\), depending on
\(K,N\) and called the profile function of \((\X,\sfd,\mm)\),
such that \(C({\rm L})\to 1\) as \({\rm L}\to 0\) and with
the following property:  given any \(\infty\)-test plan
\(\eeta\in\Pi_\infty(\X)\) with bounded support, there exists a
\(\infty\)-optimal dynamical plan
\(\ppi\in{\rm OptGeo}_\infty\big((\e_0)_\#\eeta,(\e_1)_\#\eeta\big)\)
such that
\begin{equation}\label{eq:est_Comp}
{\rm Comp}(\ppi)\leq C\big({\rm Lip}(\eeta)\big){\rm Comp}(\eeta).
\end{equation}
Moreover, in the case \(K\geq 0\), the profile function can be
additionally required to satisfy \(C\leq 1\) on a right neighbourhood
of \(0\).
\end{theorem}
\begin{proof}
Fix any boundedly-supported \(\infty\)-test plan \(\eeta\) on
\((\X,\sfd,\mm)\). Call \(\mu_0=\rho_0\mm\coloneqq(\e_0)_\#\eeta\)
and \(\mu_1=\rho_1\mm\coloneqq(\e_1)_\#\eeta\). We know from
Theorem \ref{thm:equiv_CD_q} that \((\X,\sfd,\mm)\) is a (non-branching)
\({\sf CD}_q(K,N)\) space for every \(q\in(1,\infty)\), thus
Proposition \ref{prop:bounds_for_OptGeo} ensures that for any \(n\in\N\)
the unique element \(\ppi_n\) of \({\rm OptGeo}_n(\mu_0,\mu_1)\) satisfies
\begin{equation}\label{eq:RCD_BIP_aux1}
\rho^n_t(\gamma_t)^{-\frac{1}{N}}\geq
\rho_0(\gamma_0)^{-\frac{1}{N}}\tau^{(1-t)}_{K,N}
\big(\sfd(\gamma_0,\gamma_1)\big)+\rho_1(\gamma_1)^{-\frac{1}{N}}
\tau^{(t)}_{K,N}\big(\sfd(\gamma_0,\gamma_1)\big),
\end{equation}
for \(\ppi_n\)-a.e.\ \(\gamma\) and every \(t\in[0,1]\),
where we write \((\e_t)_\#\ppi_n=\rho^n_t\mm\). Observe that
\[
\lim_{\theta\searrow 0}\frac{\tau^{(t)}_{K,N}(\theta)}{t}=1,
\quad\text{ uniformly in }t\in(0,1).
\]
In particular, the non-decreasing, continuous function
\(C\colon(0,+\infty)\to(0,+\infty)\), given by
\[
C({\rm L})\coloneqq\bigg(\sup_{\theta\in(0,{\rm L}]}
\sup_{t\in(0,1)}\frac{t}{\tau^{(t)}_{K,N}(\theta)}\bigg)^N,
\quad\text{ for every }{\rm L}>0,
\]
converges to \(1\) as \({\rm L}\to 0\). Notice that if \(K\geq 0\),
then \(C\leq 1\) on a right neighbourhood of \(0\).
Now define \(\varepsilon_n\coloneqq{\rm Lip}(\eeta)(\sqrt[n]{n}-1)\)
for every \(n\in\N\) and observe that \(\varepsilon_n\to 0\) as
\(n\to\infty\). Denote
\[
\Gamma_n\coloneqq\bigg\{\gamma\in{\rm LIP}([0,1],\X)\,:\,\int_0^1
|\dot\gamma_t|^n\,\d t\leq\big({\rm Lip}(\eeta)+\varepsilon_n\big)^n\bigg\}.
\]
Standard verifications show that \(\Gamma_n\) is a Borel subset
of \(C([0,1],\X)\). Given that
\[\begin{split}
{\rm Ke}_n(\ppi_n)&\overset{\eqref{eq:dynOPq}}=W_n^n(\mu_0,\mu_1)
\leq\int\sfd^n(x,y)\,\d(\e_0,\e_1)_\#\eeta(x,y)=
\int\sfd^n(\gamma_0,\gamma_1)\,\d\eeta(\gamma)\\
&\overset{\phantom{\eqref{eq:dynOPq}}}\leq
\int\bigg(\int_0^1|\dot\gamma_t|\,\d t\bigg)^n\d\eeta(\gamma)
\leq\int\!\!\!\int_0^1|\dot\gamma_t|^n\,\d t\,\d\eeta(\gamma)
={\rm Ke}_n(\eeta)\leq{\rm Lip}(\eeta)^n,
\end{split}\]
an application of Chebyshev's inequality yields
\[
\ppi_n(\Gamma_n^c)\leq\frac{1}{({\rm Lip}(\eeta)+\varepsilon_n)^n}
\int\!\!\!\int_0^1|\dot\gamma_t|^n\,\d t\,\d\ppi_n(\gamma)=
\frac{{\rm Ke}_n(\ppi_n)}{({\rm Lip}(\eeta)+\varepsilon_n)^n}\leq
\bigg(\frac{{\rm Lip}(\eeta)}{{\rm Lip}(\eeta)+\varepsilon_n}\bigg)^n
=\frac{1}{n},
\]
so that \(\ppi_n(\Gamma_n)\to 1\) as \(n\to\infty\).
Calling \(M\coloneqq\max\big\{\|\rho_0\|_{L^\infty(\mm)},
\|\rho_1\|_{L^\infty(\mm)}\big\}\leq{\rm Comp}(\eeta)\),
we deduce from \eqref{eq:RCD_BIP_aux1} that
\begin{equation}\label{eq:RCD_BIP_aux2}\begin{split}
\rho^n_t(\gamma_t)^{\frac{1}{N}}&\leq
\bigg(\rho_0(\gamma_0)^{-\frac{1}{N}}(1-t)\,
C\big(\sfd(\gamma_0,\gamma_1)\big)^{-\frac{1}{N}}+
\rho_1(\gamma_1)^{-\frac{1}{N}}t\,
C\big(\sfd(\gamma_0,\gamma_1)\big)^{-\frac{1}{N}}\bigg)^{-1}\\
&\leq C\big(\sfd(\gamma_0,\gamma_1)\big)^{\frac{1}{N}}M^{\frac{1}{N}},
\quad\text{ for }\ppi_n\text{-a.e.\ }\gamma\text{ and for every }t\in[0,1].
\end{split}\end{equation}
If we denote by \(D\) the diameter of
\({\rm spt}(\mu_0)\cup{\rm spt}(\mu_1)\), then
\(\ppi_n\)-a.e.\ \(\gamma\) satisfies \({\rm Lip}(\gamma)\leq D\).
Indeed, \(\ppi_n\)-a.e.\ \(\gamma\) is a geodesic joining
\(\gamma_0\in{\rm spt}(\mu_0)\) to \(\gamma_1\in{\rm spt}(\mu_1)\),
so \({\rm Lip}(\gamma)=\sfd(\gamma_0,\gamma_1)\leq D\).
In particular, \eqref{eq:RCD_BIP_aux2} implies that
\((\ppi_n)_{n\in\N}\subseteq\Pi_\infty(\X)\) and
\(\sup_{n\in\N}{\rm Comp}(\ppi_n)\leq C(D)M\). Hence, an application of
Proposition \ref{prop:existence_optgeo_infty} yields the existence
of a \(\infty\)-test plan \(\ppi\in{\rm OptGeo}_\infty(\mu_0,\mu_1)\)
such that \(\ppi_n\rightharpoonup\ppi\) as \(n\to\infty\), up to a not
relabelled subsequence. Now let us define
\[
\tilde\ppi_n\coloneqq\frac{\ppi_n|_{\Gamma_n}}{\ppi_n(\Gamma_n)}
\in\Pi_\infty(\X),\quad\text{ for every }n\in\N.
\]
Observe that \(\tilde\ppi_n\rightharpoonup\ppi\) as \(n\to\infty\)
and that, writing \((\e_t)_\#\tilde\ppi_n=\tilde\rho^n_t\mm\),
it holds \(\tilde\rho^n_t\leq\rho^n_t/\ppi_n(\Gamma_n)\).
For \(\tilde\ppi_n\)-a.e.\ \(\gamma\) one has
\(\sfd(\gamma_0,\gamma_1)\leq\int_0^1|\dot\gamma_t|\,\d t
\leq\big(\int_0^1|\dot\gamma_t|^n\,\d t\big)^{1/n}\leq{\rm Lip}(\eeta)
+\varepsilon_n\), so \eqref{eq:RCD_BIP_aux2} yields
\[
\tilde\rho^n_t(\gamma_t)\leq
\frac{C\big({\rm Lip}(\eeta)+\varepsilon_n\big)M}{\ppi_n(\Gamma_n)},
\quad\text{ for }\tilde\ppi_n\text{-a.e.\ }\gamma
\text{ and for every }t\in[0,1].
\]
This implies that \({\rm Comp}(\tilde\ppi_n)\leq C\big({\rm Lip}(\eeta)
+\varepsilon_n\big)M/\ppi_n(\Gamma_n)\), thus it
follows from Remark \ref{rmk:cont_eval_meas} that
\({\rm Comp}(\ppi)\leq C\big({\rm Lip}(\eeta)\big)M\),
as desired. Therefore, the statement is achieved.
\end{proof}
Next we introduce and study the concept of a polygonal interpolation
of a given test plan by piecewise \(\infty\)-optimal dynamical plans.
\begin{definition}\label{def:piec_optgeo}
Let \((\X,\sfd,\mm)\) be a metric measure space.
Then we define
\[\begin{split}
\Pi_{bs}(\X)&\coloneqq\big\{\ppi\in\Pi_\infty(\X)\,:
\,{\rm spt}(\ppi)\text{ is bounded}\big\},\\
\Pi_{\rm G}(\X)&\coloneqq\Big\{\ppi\in\Pi_{bs}(\X)\,:\,
\ppi\in{\rm OptGeo}_\infty\big((\e_0)_\#\ppi,(\e_1)_\#\ppi\big)\Big\}.
\end{split}\]
Moreover, we define \(\Pi_{\rm PG}(\X)\) as the family of
all those \(\infty\)-test plans \(\ppi\) on \((\X,\sfd,\mm)\)
for which there exist \(0=t_0<t_1<\ldots<t_n=1\) such that
\(({\rm restr}_{t_{i-1}}^{t_i})_\#\ppi\in\Pi_{\rm G}(\X)\) for
every \(i=1,\ldots,n\).
\end{definition}
We will say that \(\ppi\in\Pi_{\rm PG}(\X)\) is a polygonal interpolation
of a given \(\eeta\in\Pi_{bs}(\X)\) provided \((\e_{t_i})_\#\ppi=
(\e_{t_i})_\#\eeta\) for every \(i=1,\ldots,n\), where
\(t_0,\ldots,t_n\) are chosen as in Definition \ref{def:piec_optgeo}.
\begin{remark}\label{rmk:poly_interp_ineq}{\rm
We claim that
\[
{\rm Lip}(\ppi)\leq{\rm Lip}(\eeta),
\quad\text{ whenever }\ppi\in\Pi_{\rm PG}(\X)
\text{ is a polygonal interpolation of }\eeta\in\Pi_{bs}(\X).
\]
To prove it, call
\(\ppi_i\coloneqq({\rm restr}_{t_{i-1}}^{t_i})_\#\ppi\)
and \(\eeta_i\coloneqq({\rm restr}_{t_{i-1}}^{t_i})_\#\eeta\)
for every \(i=1,\ldots,n\), where \(t_0,\ldots,t_n\) are as
in Definition \ref{def:piec_optgeo}. Since \((\e_0,\e_1)_\#\eeta_i
\in{\rm Adm}\big((\e_0)_\#\ppi_i,(\e_1)_\#\ppi_i\big)\), one has that
\[\begin{split}
{\rm Lip}(\ppi_i)&=W_\infty\big((\e_0)_\#\ppi_i,(\e_1)_\#\ppi_i\big)
\leq\underset{\seeta_i\text{-a.e.\ }\gamma}{\rm ess\,sup}
\,\sfd(\gamma_0,\gamma_1)\leq\underset{\seeta_i\text{-a.e.\ }\gamma}
{\rm ess\,sup}\,{\rm Lip}(\gamma)={\rm Lip}(\eeta_i)\\
&\leq(t_i-t_{i-1}){\rm Lip}(\eeta).
\end{split}\]
Hence, we conclude that
\({\rm Lip}(\ppi)=\max_{i=1,\ldots,n}{\rm Lip}(\ppi_i)/(t_i-t_{i-1})
\leq{\rm Lip}(\eeta)\), as claimed.
\fr}\end{remark}

Let us introduce some further notation.
Given a \(\infty\)-test plan \(\ppi\),
we define its trace \([\ppi]\) as
\[
[\ppi]\coloneqq\bigcup_{t\in[0,1]}{\rm spt}
\big((\e_t)_\#\ppi\big)\subseteq\X.
\]
Moreover, given an open set \(\Omega\subseteq\X\) and a
family \(\Pi\) of \(\infty\)-test plans on \((\X,\sfd,\mm)\),
we define
\[
\Pi[\Omega]\coloneqq\Big\{\ppi\in\Pi\;\Big|\;[\ppi]\subseteq
\Omega,\,\sfd\big([\ppi],\X\setminus\Omega\big)>0\Big\}.
\]
\begin{lemma}\label{lem:aux_interp}
Let \((\X,\sfd,\mm)\) be a non-branching \({\sf CD}(K,N)\) space,
with \(K\in\R\), \(N\in(1,\infty)\), and \(\mm(\X)<+\infty\).
Let \(\eeta\in\Pi_{bs}(\X)\) be given. Then there exists a
sequence \((\ppi_n)_{n\in\N}\subseteq\Pi_{\rm PG}(\X)\)
of polygonal interpolations of \(\eeta\) such that
\[
{\rm Lip}(\ppi_n)\leq{\rm Lip}(\eeta),\quad
{\rm Comp}(\ppi_n)\leq C\big({\rm Lip}(\eeta)/n\big){\rm Comp}(\eeta),
\quad\text{ for every }n\in\N,
\]
where \({\rm L}\mapsto C({\rm L})\) stands for the profile function
of \((\X,\sfd,\mm)\).
Moreover, we can additionally require that each trace \([\ppi_n]\) is
contained in the closed \(\frac{{\rm Lip}(\seeta)}{n}\)-neighbourhood
of \([\eeta]\).
\end{lemma}
\begin{proof}
Let \(n\in\N\) be fixed. Given any \(i=1,\ldots,n\),
choose any test plan \(\ppi^i_n\in\Pi_{\rm G}(\X)\) such that
\((\e_0)_\#\ppi^i_n=(\e_{(i-1)/n})_\#\eeta\),
\((\e_1)_\#\ppi^i_n=(\e_{i/n})_\#\eeta\), and
\begin{equation}\label{eq:aux_interp_aux}
{\rm Comp}(\ppi^i_n)\leq C\big({\rm Lip}(\eeta)/n\big){\rm Comp}(\eeta),
\end{equation}
whose existence is guaranteed by Theorem \ref{thm:CD_are_BIP}.
Thanks to a glueing argument, we find a plan
\(\ppi_n\in\Pi_{\rm PG}(\X)\) such that
\(({\rm restr}_{(i-1)/n}^{i/n})_\#\ppi_n=\ppi_n^i\) for every
\(i=1,\ldots,n\). Note that \eqref{eq:aux_interp_aux} yields
\({\rm Comp}(\ppi_n)\leq C\big({\rm Lip}(\eeta)/n\big){\rm Comp}(\eeta)\).
Also, \(\ppi_n\) is a polygonal interpolation of \(\eeta\) by construction,
thus \({\rm Lip}(\ppi_n)\leq{\rm Lip}(\eeta)\)
by Remark \ref{rmk:poly_interp_ineq}.
Finally, since each \(\infty\)-test plan \(\ppi^i_n\) satisfies
\({\rm Lip}(\ppi^i_n)\leq{\rm Lip}(\eeta)/n\),
we deduce that \([\ppi^i_n]\) lies inside the closed
\(\frac{{\rm Lip}(\seeta)}{n}\)-neighbourhood of
\({\rm spt}\big((\e_{i/n})_\#\eeta\big)\), and thus accordingly
\([\ppi_n]\) lies in the
\(\frac{{\rm Lip}(\seeta)}{n}\)-neighbourhood of \([\eeta]\).
\end{proof}

Before passing to the main result of this section,
we introduce some auxiliary terminology.
Given any non-empty family \(\Pi\) of \(\infty\)-test plans
on a metric measure space \((\X,\sfd,\mm)\), we denote by
\({\rm BV}_\Pi^\star(\X)\) the family of all functions
\(f\in L^1(\mm)\) for which there exists \(C\geq 0\) satisfying
\[
\int f(\gamma_1)-f(\gamma_0)\,\d\ppi(\gamma)\leq
{\rm Comp}(\ppi){\rm Lip}(\ppi)C,\quad\text{ for every }\ppi\in\Pi.
\]
The minimal such constant \(C\geq 0\) will be denoted by
\(|{\boldsymbol D}f|_\Pi^\star(\X)\). Notice that Theorem
\ref{thm:equiv_BV} can be equivalently rephrased by saying that
\({\rm BV}(\X)={\rm BV}_{\Pi_\infty}^\star(\X)\) and that
\(|{\boldsymbol D}f|(\X)=|{\boldsymbol D}f|_{\Pi_\infty}^\star(\X)\) for
every \(f\in{\rm BV}(\X)\), where we set
\(\Pi_\infty\coloneqq\Pi_\infty(\X)\) for brevity.
For any \(f\in{\rm BV}_\Pi^\star(\X)\), we define
\[
|{\boldsymbol D}f|_\Pi^\star(\Omega)\coloneqq
\sup_{\sppi\in\Pi[\Omega]}\frac{1}{{\rm Comp}(\ppi){\rm Lip}(\ppi)}
\int f(\gamma_1)-f(\gamma_0)\,\d\ppi,\quad\text{ for every }
\Omega\subseteq\X\text{ open,}
\]
and we extend it to the whole Borel \(\sigma\)-algebra of
\((\X,\sfd)\) in the following standard manner:
\[
|{\boldsymbol D}f|_\Pi^\star(B)\coloneqq
\inf\big\{|{\boldsymbol D}f|_\Pi^\star(\Omega)\;\big|\;
\Omega\subseteq\X\text{ open, }B\subseteq\Omega\big\},
\quad\text{ for every }B\subseteq\X\text{ Borel.}
\]
Notice that we are not claiming that \(|{\boldsymbol D}f|_\Pi^\star\)
is a measure for an arbitrary choice of \(\Pi\). Observe also
that \({\rm BV}_\Pi(\X)\subseteq{\rm BV}_\Pi^\star(\X)\) and
\(|{\boldsymbol D}f|_\Pi^\star(B)\leq|{\boldsymbol D}f|_\Pi(B)\)
for every \(f\in{\rm BV}_\Pi(\X)\) and \(B\subseteq\X\) Borel, as
one can easily check directly from the definitions. It also holds
that \(|{\boldsymbol D}f|=|{\boldsymbol D}f|_{\Pi_\infty}^\star\).
\begin{theorem}[Master geodesic plans on non-branching \(\sf CD\) spaces]
\label{thm:master_G}
Let \((\X,\sfd,\mm)\) be a non-branching \({\sf CD}(K,N)\) space,
for some \(K\in\R\) and \(N\in(1,\infty)\). Suppose the measure
\(\mm\) is finite. Then \(\Pi_{\rm G}(\X)\) is a master family for
\({\rm BV}(\X)\). More generally,
\({\rm BV}(\X)={\rm BV}_{\Pi_{\rm G}(\X)}^\star(\X)\) and
\[
|{\boldsymbol D}f|(B)=|{\boldsymbol D}f|_{\Pi_{\rm G}(\X)}^\star(B),
\quad\text{ for every }f\in{\rm BV}(\X)\text{ and }B\subseteq\X
\text{ Borel.}
\]
\end{theorem}
\begin{proof}
For the sake of brevity, we will write \(\Pi_{\rm G}\) in place
of \(\Pi_{\rm G}(\X)\). It follows from the very definitions
of the involved objects that
\({\rm BV}_{\Pi_\infty}^\star(\X)={\rm BV}(\X)\subseteq
{\rm BV}_{\Pi_{\rm G}}(\X)\subseteq{\rm BV}_{\Pi_{\rm G}}^\star(\X)\) and
\[
|{\boldsymbol D}f|_{\Pi_{\rm G}}^\star(B)
\leq|{\boldsymbol D}f|_{\Pi_{\rm G}}(B)\leq
|{\boldsymbol D}f|(B)=|{\boldsymbol D}f|_{\Pi_\infty}^\star(B),
\quad\text{ for all }f\in{\rm BV}(\X)\text{ and }B\subseteq\X
\text{ Borel.}
\]
Hence, in order to achieve the statement it suffices to prove that
\({\rm BV}_{\Pi_{\rm G}}^\star(\X)\subseteq
{\rm BV}_{\Pi_\infty}^\star(\X)\) and
\begin{equation}\label{eq:master_G_claim1}
|{\boldsymbol D}f|_{\Pi_\infty}^\star(\Omega)\leq|{\boldsymbol D}f|
_{\Pi_{\rm G}}^\star(\Omega)\quad\text{for every }f\in{\rm BV}
_{\Pi_{\rm G}}^\star(\X)\text{ and }\Omega\subseteq\X\text{ open.}
\end{equation}
{\color{blue}\textsc{Step 1.}} 
Let \(\Omega\subseteq\X\) be a given open set. Calling
\(\Pi_{\rm PG}\coloneqq\Pi_{\rm PG}(\X)\) for brevity, we claim that
\begin{equation}\label{eq:master_G_claim2}
{\rm BV}_{\Pi_{\rm G}}^\star(\X)\subseteq
{\rm BV}_{\Pi_{\rm PG}}^\star(\X),\qquad
|{\boldsymbol D}f|_{\Pi_{\rm PG}}^\star(\Omega)
\leq|{\boldsymbol D}f|_{\Pi_{\rm G}}^\star(\Omega)
\quad\text{for every }f\in{\rm BV}_{\Pi_{\rm G}}^\star(\X).
\end{equation}
In order to prove it, fix \(f\in{\rm BV}_{\Pi_{\rm G}}^\star(\X)\)
and \(\ppi\in\Pi_{\rm PG}[\Omega]\). Choose \(0=t_0<t_1<\ldots<t_n=1\) such that
\(\ppi_i\coloneqq({\rm restr}_{t_{i-1}}^{t_i})_\#\ppi\in\Pi_{\rm G}[\Omega]\)
for all \(i=1,\ldots,n\). Since \({\rm Lip}(\ppi_i)\leq(t_i-t_{i-1})
{\rm Lip}(\ppi)\) and \({\rm Comp}(\ppi_i)\leq{\rm Comp}(\ppi)\) for
every \(i=1,\ldots,n\), we may estimate
\[\begin{split}
\int f(\gamma_1)-f(\gamma_0)\,\d\ppi(\gamma)&=
\sum_{i=1}^n\int f(\gamma_{t_i})-f(\gamma_{t_{i-1}})\,\d\ppi(\gamma)
=\sum_{i=1}^n\int f(\gamma_1)-f(\gamma_0)\,\d\ppi_i(\gamma)\\
&\leq\sum_{i=1}^n{\rm Comp}(\ppi_i){\rm Lip}(\ppi_i)
|{\boldsymbol D}f|_{\Pi_{\rm G}}^\star(\Omega)\leq{\rm Comp}(\ppi)
{\rm Lip}(\ppi)|{\boldsymbol D}f|_{\Pi_{\rm G}}^\star(\Omega),
\end{split}\]
whence \eqref{eq:master_G_claim2} follows thanks to the arbitrariness
of \(\ppi\in\Pi_{\rm PG}[\Omega]\).\\
{\color{blue}\textsc{Step 2.}} Next we claim that, calling
\(\Pi_{bs}\coloneqq\Pi_{bs}(\X)\) for brevity, it holds that
\begin{equation}\label{eq:master_G_claim3}
{\rm BV}_{\Pi_{\rm PG}}^\star(\X)\subseteq{\rm BV}_{\Pi_{bs}}^\star(\X),
\qquad|{\boldsymbol D}f|_{\Pi_{bs}}^\star(\Omega)\leq
|{\boldsymbol D}f|_{\Pi_{\rm PG}}^\star(\Omega)
\quad\text{for every }f\in{\rm BV}_{\Pi_{\rm PG}}^\star(\X).
\end{equation}
In order to prove it, fix \(f\in{\rm BV}_{\Pi_{\rm PG}}^\star(\X)\)
and \(\ppi\in\Pi_{bs}[\Omega]\). Lemma \ref{lem:aux_interp} yields the existence
of polygonal interpolations \((\ppi_n)_{n\in\N}\subseteq\Pi_{\rm PG}\)
of \(\ppi\) with \({\rm Comp}(\ppi_n)\leq
C\big({\rm Lip}(\ppi)/n\big){\rm Comp}(\ppi)\) and
\({\rm Lip}(\ppi_n)\leq{\rm Lip}(\ppi)\) for every \(n\in\N\),
and such that \([\ppi_n]\) lies in the closed
\(\frac{{\rm Lip}(\sppi)}{n}\)-neighbourhood of \([\ppi]\).
Chosen \(\bar n\in\N\) so that \({\rm Lip}(\ppi)/\bar n<
\sfd\big([\ppi],\X\setminus\Omega\big)\), we thus have that
\(\ppi_n\in\Pi_{\rm PG}[\Omega]\) for every \(n\geq\bar n\).
Given that \(C({\rm L})\to 1\) as \({\rm L}\to 0\), by letting
\(n\to\infty\) in
\[\begin{split}
\int f(\gamma_1)-f(\gamma_0)\,\d\ppi(\gamma)
&=\int f\,\d(\e_1)_\#\ppi-\int f\,\d(\e_0)_\#\ppi
=\int f\,\d(\e_1)_\#\ppi_n-\int f\,\d(\e_0)_\#\ppi_n\\
&=\int f(\gamma_1)-f(\gamma_0)\,\d\ppi_n(\gamma)
\leq{\rm Comp}(\ppi_n){\rm Lip}(\ppi_n)
|{\boldsymbol D}f|_{\Pi_{\rm PG}}^\star(\Omega)\\
&\leq C\big({\rm Lip}(\ppi)/n\big){\rm Comp}(\ppi){\rm Lip}(\ppi)
|{\boldsymbol D}f|_{\Pi_{\rm PG}}^\star(\Omega)
\end{split}\]
we conclude that \(\int f(\gamma_1)-f(\gamma_0)\,\d\ppi(\gamma)\leq
{\rm Comp}(\ppi){\rm Lip}(\ppi)|{\boldsymbol D}f|_{\Pi_{\rm PG}}^\star(\Omega)\),
obtaining \eqref{eq:master_G_claim3}.\\
{\color{blue}\textsc{Step 3.}} Finally, we claim that
\begin{equation}\label{eq:master_G_claim4}
{\rm BV}_{\Pi_{bs}}^\star(\X)\subseteq{\rm BV}_{\Pi_\infty}^\star(\X),
\qquad|{\boldsymbol D}f|_{\Pi_\infty}^\star(\Omega)\leq|{\boldsymbol D}f|
_{\Pi_{bs}}^\star(\Omega)\quad\text{for every }f\in{\rm BV}_{\Pi_{bs}}^\star(\X).
\end{equation}
In order to prove it, fix any \(f\in{\rm BV}_{\Pi_{bs}}^\star(\X)\)
and a \(\infty\)-test plan \(\ppi\in\Pi_\infty[\Omega]\) on \((\X,\sfd,\mm)\).
Fix a curve \(\bar\gamma\in{\rm spt}(\ppi)\) and define
\(\Gamma_n\coloneqq\big\{\gamma\in C([0,1],\X)\,:\,
\sfd_{\rm sup}(\gamma,\bar\gamma)\leq n\big\}\) for every \(n\in\N\).
Define
\[
\ppi_n\coloneqq\frac{\ppi|_{\Gamma_n}}{\ppi(\Gamma_n)}
\in\Pi_{bs}[\Omega],\quad\text{ for every }n\in\N.
\]
Observe that \({\rm Comp}(\ppi_n)\leq{\rm Comp}(\ppi)/\ppi(\Gamma_n)\)
and \({\rm Lip}(\ppi_n)\leq{\rm Lip}(\ppi)\) for every \(n\in\N\).
Also, it holds \(\ppi(\Gamma_n)\nearrow 1\) as \(n\to\infty\).
Therefore, by using the dominated convergence theorem we get
\[\begin{split}
\int f(\gamma_1)-f(\gamma_0)\,\d\ppi(\gamma)&=
\lim_{n\to\infty}\int f(\gamma_1)-f(\gamma_0)\,\d\ppi_n(\gamma)\\
&\leq\limi_{n\to\infty}{\rm Comp}(\ppi_n){\rm Lip}(\ppi_n)
|{\boldsymbol D}f|_{\Pi_{bs}}^\star(\Omega)\\
&\leq{\rm Comp}(\ppi){\rm Lip}(\ppi)|{\boldsymbol D}f|
_{\Pi_{bs}}^\star(\Omega)\lim_{n\to\infty}\frac{1}{\ppi(\Gamma_n)}\\
&={\rm Comp}(\ppi){\rm Lip}(\ppi)|{\boldsymbol D}f|_{\Pi_{bs}}^\star(\Omega),
\end{split}\]
thus proving \eqref{eq:master_G_claim4}. By combining
\eqref{eq:master_G_claim2}, \eqref{eq:master_G_claim3},
and \eqref{eq:master_G_claim4}, we eventually obtain
\eqref{eq:master_G_claim1}.
\end{proof}
\begin{remark}\label{rmk:simpler_K_positive}{\rm
We point out that if \((\X,\sfd,\mm)\) is a non-branching
\({\sf CD}(K,N)\) space with \(K\geq 0\), then the proof of Theorem
\ref{thm:master_G} can be significantly simplified. Indeed, in this
case one can prove that \({\rm BV}_{\Pi_{\rm G}}^\star(\X)
\subseteq{\rm BV}_{\Pi_{bs}}^\star(\X)\) and
\(|{\boldsymbol D}f|_{\Pi_{bs}}^\star(\X)\leq|{\boldsymbol D}f|_{\Pi
_{\rm G}}^\star(\X)\) for every \(f\in{\rm BV}_{\Pi_{\rm G}}^\star(\X)\)
in the following way. Given any \(f\in{\rm BV}_{\Pi_{\rm G}}^\star(\X)\)
and \(\eeta\in\Pi_{bs}\), let us just take the \(\infty\)-test plan
\(\ppi\in{\rm OptGeo}_\infty\big((\e_0)_\#\eeta,(\e_1)_\#\eeta\big)\)
satisfying both \({\rm Comp}(\ppi)\leq{\rm Comp}(\eeta)\)
and \({\rm Lip}(\ppi)\leq{\rm Lip}(\eeta)\), whose existence is
granted by Theorem \ref{thm:CD_are_BIP}. Then
\[\begin{split}
\int f(\gamma_1)-f(\gamma_0)\,\d\eeta(\gamma)&=
\int f(\gamma_1)-f(\gamma_0)\,\d\ppi(\gamma)\leq
{\rm Comp}(\ppi){\rm Lip}(\ppi)|{\boldsymbol D}f|_{\Pi_{\rm G}}^\star(\X)\\
&\leq{\rm Comp}(\eeta){\rm Lip}(\eeta)|{\boldsymbol D}f|_{\Pi_{\rm G}}^\star(\X),
\end{split}\]
thus proving that \(f\in{\rm BV}_{\Pi_{bs}}^\star(\X)\) and
\(|{\boldsymbol D}f|_{\Pi_{bs}}^\star(\X)\leq|{\boldsymbol D}f|_{\Pi_{\rm G}}^\star(\X)\).
\fr}\end{remark}

In the more general setting of non-branching \(\sf MCP\) spaces,
we can prove a weaker statement:
\begin{remark}[A weaker form of master geodesic plans
on \(\sf MCP\) spaces]\label{rmk:MCP}{\rm
Let \((\X,\sfd,\mm)\) be a non-branching \({\sf MCP}(K,N)\) metric
measure space with \(N<\infty\) (but without finitess assumptions
on \(\mm\)). We recall that in the non-branching case the
\({\sf MCP}(K,N)\) condition (standing for Measure Contraction
Property), which was introduced independently by Ohta \cite{Ohta07}
and Sturm \cite{Sturm06II}, is weaker than \({\sf CD}_q(K,N)\) for
any \(q\in(1,\infty)\).
By adapting our previous arguments, we can get
\({\rm BV}_{\Pi_{\rm G}}(\X)={\rm BV}(\X)\) and
\begin{equation}\label{eq:claim_MCP}
|{\boldsymbol D}f|_{\Pi_{\rm G}}(B)\leq|{\boldsymbol D}f|(B)\leq 2^N
|{\boldsymbol D}f|_{\Pi_{\rm G}}(B),\quad\text{ for every }f\in{\rm BV}(\X)\text{ and }B\subseteq\X\text{ Borel.}
\end{equation}
The first inequality is always verified.
Below we sketch the proof of the second inequality.

Fix a boundedly-supported \(\eeta\in\Pi_\infty(\X)\) and
\(q\in(1,\infty)\). Call \(\mu_i\coloneqq(\e_i)_\#\eeta\)
for \(i=0,1\). We know from \cite[Proposition 9.1]{CM16} that the
unique element \(\ppi_q\) of \({\rm OptGeo}_q(\mu_0,\mu_1)\) satisfies
\begin{equation}\label{eq:MCP_aux1}
\rho^q_t(\gamma_t)^{-\frac{1}{N}}\geq
\rho_0(\gamma_0)^{-\frac{1}{N}}\tau_{K,N}^{(1-t)}
\big(\sfd(\gamma_0,\gamma_1)\big),\quad\text{ for }
\ppi_q\text{-a.e.\ }\gamma\text{ and every }t\in[0,1),
\end{equation}
where we set \(\rho^q_t\mm\coloneqq(\e_t)_\#\ppi_q\).
We point out that \cite[Proposition 9.1]{CM16} concerns the case
\(q=2\), but the proof argument works for \(q\in(1,\infty)\)
arbitrary (recall that the \(\sf MCP\) condition is by definition
independent of \(q\)). By applying \eqref{eq:MCP_aux1} to the
`reversed-in-time' plan \(({\rm restr}_1^0)_\#\ppi_q\), which
is the unique element of \({\rm OptGeo}_q(\mu_1,\mu_0)\), we
obtain the symmetric estimate
\begin{equation}\label{eq:MCP_aux2}
\rho^q_t(\gamma_t)^{-\frac{1}{N}}\geq
\rho_1(\gamma_1)^{-\frac{1}{N}}\tau_{K,N}^{(t)}
\big(\sfd(\gamma_0,\gamma_1)\big),\quad\text{ for }
\ppi_q\text{-a.e.\ }\gamma\text{ and every }t\in(0,1].
\end{equation}
Define the function \(C\colon(0,+\infty)\to(0,+\infty)\)
as in the proof of Theorem \ref{thm:CD_are_BIP}. Then we have
\[
\rho^q_t(\gamma_t)^{\frac{1}{N}}\overset{\eqref{eq:MCP_aux1}}\leq
\frac{1}{1-t}\,C\big(\sfd(\gamma_0,\gamma_1)\big)^{\frac{1}{N}}
\rho_0(\gamma_0)^{\frac{1}{N}}\leq 2\,C\big(\sfd(\gamma_0,\gamma_1)
\big)^{\frac{1}{N}}{\rm Comp}(\eeta)^{\frac{1}{N}},
\]
for \(\ppi_q\)-a.e.\ \(\gamma\) and every \(t\in[0,1/2]\).
Similarly, we get \(\rho^q_t(\gamma_t)^{\frac{1}{N}}\leq
2\,C\big(\sfd(\gamma_0,\gamma_1)\big)^{\frac{1}{N}}
{\rm Comp}(\eeta)^{\frac{1}{N}}\) for \(\ppi_q\)-a.e.\ \(\gamma\)
and every \(t\in[1/2,1]\) by using \eqref{eq:MCP_aux2} in place
of \eqref{eq:MCP_aux1}. By arguing as in the last part of the proof
of Theorem \ref{thm:CD_are_BIP} (letting \(q\to\infty\) and using the
compactness result in Proposition \ref{prop:compactness_tp}),
we thus obtain a plan \(\ppi\in{\rm OptGeo}_\infty(\mu_0,\mu_1)\)
such that
\begin{equation}\label{eq:alt_est_MCP}
{\rm Comp}(\ppi)\leq 2^N C\big({\rm Lip}(\eeta)\big){\rm Comp}(\eeta).
\end{equation} 
Finally, by arguing as in the proof of Theorem \ref{thm:master_G},
but using \eqref{eq:alt_est_MCP} in place of \eqref{eq:est_Comp},
we conclude that
\(|{\boldsymbol D}f|(\Omega)=|{\boldsymbol D}f|_{\Pi_\infty}^\star
(\Omega)\leq 2^N|{\boldsymbol D}f|_{\Pi_{\rm G}}^\star(\Omega)\leq 2^N
|{\boldsymbol D}f|_{\Pi_{\rm G}}(\Omega)\) holds for every choice of
\(f\in{\rm BV}_{\Pi_{\rm G}}(\X)\) and \(\Omega\subseteq\X\) open,
whence \eqref{eq:claim_MCP} follows by outer regularity.
\fr}\end{remark}
\subsection{Master sequences of test plans on metric measure spaces}
\label{ss:master_seq}
Aim of this section is to prove that on any metric measure
space one can find a countable master family for BV.
On non-branching \(\sf CD\) spaces, the master sequence can
be required to consist of geodesic plans.
\begin{theorem}\label{thm:countable_master_tp}
Let \((\X,\sfd,\mm)\) be a metric measure space. Then there
exists an (at most) countable family \(\Pi\subseteq\Pi_\infty(\X)\)
of \(\infty\)-test plans on \((\X,\sfd,\mm)\) that is a master family
for \({\rm BV}(\X)\).

If \((\X,\sfd,\mm)\) is a non-branching \({\sf CD}(K,N)\)
space with \(K\in\R\), \(N\in(1,\infty)\), and \(\mm(\X)<+\infty\),
then we can additionally require that \(\Pi\subseteq\Pi_{\rm G}(\X)\).
\end{theorem}
\begin{proof}
First, fix any master family \(\Pi^{\sf m}\) for \({\rm BV}(\X)\).
In particular, one can take \(\Pi^{\sf m}=\Pi_\infty(\X)\). In the
setting of non-branching \({\sf CD}(K,N)\) spaces with \(N,\mm\) finite,
we know from Theorem \ref{thm:master_G} that also the
choice \(\Pi^{\sf m}=\Pi_{\rm G}(\X)\) is allowed. Now define
\[
\Pi^{\sf m}_{\alpha,\beta}\coloneqq\Big\{\ppi\in\Pi^{\sf m}\,:\,
{\rm Comp}(\ppi)\leq\alpha,\,{\rm Lip}(\ppi)\leq\beta\Big\},
\quad\text{ for every }\alpha,\beta\in\mathbb Q^+.
\]
Note that \(\Pi^{\sf m}=\bigcup_{\alpha,\beta\in\mathbb Q^+}
\Pi^{\sf m}_{\alpha,\beta}\). Given any \(\alpha,\beta\in\mathbb Q^+\),
we select a countable set \(\mathcal C_{\alpha,\beta}
\subseteq\Pi^{\sf m}_{\alpha,\beta}\) that is dense in
\(\Pi^{\sf m}_{\alpha,\beta}\) with respect to the narrow topology.
Consider the countable family
\[
\Pi\coloneqq\bigcup_{\alpha,\beta\in\mathbb Q^+}\mathcal C_{\alpha,\beta}
\subseteq\Pi^{\sf m}.
\]
We aim to show that \(\Pi\) fulfills the requirements. Given that
\({\rm BV}_{\Pi^{\sf m}}^\star(\X)={\rm BV}(\X)\subseteq{\rm BV}_\Pi(\X)\)
and \(|{\boldsymbol D}f|_\Pi\leq|{\boldsymbol D}f|\) for every
\(f\in{\rm BV}(\X)\), it is sufficient to prove that
\({\rm BV}_\Pi(\X)\subseteq{\rm BV}_{\Pi^{\sf m}}^\star(\X)\) and
\(|{\boldsymbol D}f|_{\Pi^{\sf m}}^\star(\X)\leq|{\boldsymbol D}f|_\Pi(\X)\)
for every \(f\in{\rm BV}_\Pi(\X)\), which amounts to showing that
\begin{equation}\label{eq:countable_master_tp_aux}
\int f(\gamma_1)-f(\gamma_0)\,\d\ppi(\gamma)\leq
{\rm Comp}(\ppi){\rm Lip}(\ppi)|{\boldsymbol D}f|_\Pi(\X),\quad
\text{ for every }f\in{\rm BV}_\Pi(\X)\text{ and }\ppi\in\Pi^{\sf m}.
\end{equation}
To this aim, let \(f\in{\rm BV}_\Pi(\X)\) and \(\ppi\in\Pi^{\sf m}\)
be fixed. Pick two sequences \((\alpha_n)_{n\in\N},(\beta_n)_{n\in\N}
\subseteq\mathbb Q^+\) satisfying \(\alpha_n\searrow{\rm Comp}(\ppi)\)
and \(\beta_n\searrow{\rm Lip}(\ppi)\). For any \(n\in\N\)
we have that \(\ppi\in\Pi^{\sf m}_{\alpha_n,\beta_n}\),
so that we can find a plan \(\ppi_n\in\mathcal C_{\alpha_n,\beta_n}
\subseteq\Pi\) such that \(\sfd_{\mathscr P}(\ppi_n,\ppi)\leq 1/n\).
This means that \(\ppi_n\rightharpoonup\ppi\) as \(n\to\infty\),
so that in particular \((\e_0)_\#\ppi_n\rightharpoonup(\e_0)_\#\ppi\)
and \((\e_1)_\#\ppi_n\rightharpoonup(\e_1)_\#\ppi\) by Remark
\ref{rmk:cont_eval_meas}. Given any \(\varepsilon>0\),
we can find a function \(f_\varepsilon\in C_{bs}(\X)\) such that
\(\|f-f_\varepsilon\|_{L^1(\mm)}\leq\varepsilon\). Using
the fact that \(\sup_{n\in\N}{\rm Comp}(\ppi_n)\leq\alpha_1\),
we can thus estimate
\[\begin{split}
&\bigg|\int f(\gamma_1)-f(\gamma_0)\,\d\ppi(\gamma)-
\int f(\gamma_1)-f(\gamma_0)\,\d\ppi_n(\gamma)\bigg|\\
\leq\,&\bigg|\int f_\varepsilon(\gamma_1)-f_\varepsilon(\gamma_0)
\,\d\ppi(\gamma)-\int f_\varepsilon(\gamma_1)-f_\varepsilon(\gamma_0)
\,\d\ppi_n(\gamma)\bigg|+4\alpha_1\varepsilon\\
\leq\,&\bigg|\int f_\varepsilon\,\d(\e_1)_\#\ppi-
\int f_\varepsilon\,\d(\e_1)_\#\ppi_n\bigg|+
\bigg|\int f_\varepsilon\,\d(\e_0)_\#\ppi-
\int f_\varepsilon\,\d(\e_0)_\#\ppi_n\bigg|+4\alpha_1\varepsilon.
\end{split}\]
By first letting \(n\to\infty\) and then \(\varepsilon\searrow 0\),
we deduce that
\[
\int f(\gamma_1)-f(\gamma_0)\,\d\ppi_n(\gamma)
\to\int f(\gamma_1)-f(\gamma_0)\,\d\ppi(\gamma),
\quad\text{ as }n\to\infty.
\]
Therefore, we can eventually conclude that
\[\begin{split}
\int f(\gamma_1)-f(\gamma_0)\,\d\ppi(\gamma)&=
\lim_{n\to\infty}\int f(\gamma_1)-f(\gamma_0)\,\d\ppi_n(\gamma)\\
&\leq\limi_{n\to\infty}\int|D(f\circ\gamma)|([0,1])\,\d\ppi_n(\gamma)\\
&=\limi_{n\to\infty}\int\gamma_\#|D(f\circ\gamma)|(\X)\,\d\ppi_n(\gamma)\\
&\leq\limi_{n\to\infty}{\rm Comp}(\ppi_n){\rm Lip}(\ppi_n)
|{\boldsymbol D}f|_\Pi(\X)\\
&\leq\lim_{n\to\infty}\alpha_n\beta_n|{\boldsymbol D}f|_\Pi(\X)\\
&={\rm Comp}(\ppi){\rm Lip}(\ppi)|{\boldsymbol D}f|_\Pi(\X).
\end{split}\]
This gives the desired inequality in \eqref{eq:countable_master_tp_aux},
thus yielding the sought conclusion.
\end{proof}
\begin{remark}{\rm
Actually, the proof of Theorem \ref{thm:countable_master_tp}
shows also the following statement: from any master family
\(\Pi\) for \({\rm BV}(\X)\), one can extract a countable subfamily
\(\tilde\Pi\subseteq\Pi\) that is still a master family for
\({\rm BV}(\X)\).
\fr}\end{remark}

We expect that, since the defining property \eqref{def:BV} of
\({\rm BV}(\X)\) space is in `integral form', it is not suitable to improve Theorem
\ref{thm:countable_master_tp} and obtain a single master plan
for \({\rm BV}(\X)\), \emph{i.e.}, a master family that is a singleton. 
However, in Section \ref{s:curvewise_BV}
we will deal with a `curvewise' definition of \({\rm BV}(\X)\),
which will allow us to build a single test plan (concentrated on
geodesics in the non-branching \(\sf CD\) case) that is a master
plan for \({\rm BV}(\X)\) in the curvewise sense.
\section{Curvewise BV space}\label{s:curvewise_BV}
As we already saw during the paper, a smorgasbord
of different (and mostly equivalent) notions of BV space over
metric measure spaces have been thoroughly studied in the literature.
In this section we propose yet another definition of BV space,
which we will refer to as the curvewise BV space.
In Subsection \ref{ss:def_BVcw} we introduce it and prove its
equivalence with \({\rm BV}(\X)\). In Subsection \ref{ss:master_cw_tp}
we show that Theorem \ref{thm:countable_master_tp} implies the existence
of a single \(\infty\)-test plan which is a master test plan in the
curvewise BV sense.
\subsection{Definition of \texorpdfstring{\({\rm BV}^{\sf cw}(\X)\)}
{BVcw} and its main properties}\label{ss:def_BVcw}
Here we introduce a new notion of function of bounded variation
on a metric measure space: the \emph{curvewise BV space}, which
we shall denote by \({\rm BV}^{\sf cw}(\X)\). Our definition is
heavily inspired by the so-called \emph{AM-BV space}, which we
are going to recall briefly. The potential-theoretic notion of
\emph{approximation modulus}, AM-modulus for short, has been
recently introduced by O.\ Martio in \cite{Martio16}. By building
on top of it, he constructed in \cite{Martio16-2} a `Newtonian-type'
version of BV space, denoted by \({\rm BV}_{AM}(\X)\).

Much like an \(L^p\)-function belongs to the Newtonian--Sobolev space
\(N^{1,p}(\X)\) provided it satisfies the weak upper gradient
inequality along \({\rm Mod}_p\)-a.e.\ path (where \({\rm Mod}_p\)
stands for the \(p\)-modulus), an \(L^1\)-function is declared
to be in \({\rm BV}_{AM}(\X)\) provided it satisfies the
\emph{\(BV_{AM}\) upper bound} inequality
(cf.\ \cite[Eq.\ (2.6)]{Martio16-2}) along \(AM\)-a.e.\ path.
Our strategy is the following: to replace
the quantification `along \(AM\)-a.e.\ path' by `along
\(\ppi\)-a.e.\ curve, for every \(\infty\)-test plan \(\ppi\)'.
Technically speaking, to do so we first need to slightly adapt
the concept of \(BV_{AM}\) upper bound, thus introducing that of
\emph{curvewise bound}, see Definition \ref{def:cw_bound}.
The resulting function space \({\rm BV}^{\sf cw}(\X)\)
(cf.\ Definition \ref{def:BV_cw}) is a priori larger than
\({\rm BV}_{AM}(\X)\), the reason being that \(AM\)-null sets are
\(\ppi\)-null for every \(\infty\)-test plan \(\ppi\)
(see Lemma \ref{lem:null_AM_vs_tp}). We will prove
in Theorem \ref{thm:equiv_BV_cw} (see also Corollary
\ref{cor:equiv_BV_cw}) that \({\rm BV}^{\sf cw}(\X)={\rm BV}(\X)\)
on every metric measure space.
Furthermore, we will prove in Appendix \ref{app:AM-BV}
that also \({\rm BV}_{AM}(\X)={\rm BV}(\X)\), and thus
\({\rm BV}_{AM}(\X)={\rm BV}^{\sf cw}(\X)\),
are verified on every metric measure space.
\begin{definition}[Curvewise \(\Pi\)-bound]\label{def:cw_bound}
Let \((\X,\sfd,\mm)\) be a metric measure space.
Fix \(f\in L^1(\mm)\) and a family \(\Pi\) of \(\infty\)-test plans
on \((\X,\sfd,\mm)\). Let \(\Omega\subseteq\X\) be an open set. Then
a given sequence \((g_n)_{n\in\N}\subseteq L^1(\mm|_\Omega)\)
of non-negative functions is said to be a curvewise \(\Pi\)-bound
for \(f\) on \(\Omega\) provided for any \(\ppi\in\Pi\) the following
property is verified: for \(\ppi\)-a.e.\ curve \(\gamma\), it holds that
\begin{equation}\label{eq:def_BV_bound}
|D(f\circ\gamma)|((a,b))\leq\limi_{n\to\infty}
\int_a^b g_n(\gamma_t)|\dot\gamma_t|\,\d t,
\quad\text{ for every }0<a<b<1\text{ with }\gamma((a,b))\subseteq\Omega.
\end{equation}
When \(\Pi\) is the totality of \(\infty\)-test plans
on \((\X,\sfd,\mm)\), we just speak about curvewise bounds.
\end{definition}

The bounded compressibility assumption on the plans \(\ppi\)
ensures that \eqref{eq:def_BV_bound} is well-posed, since it is
independent of the chosen Borel representatives of \(f\) and
\((g_n)_{n\in\N}\). By building on top of the above notion of curvewise
\(\Pi\)-bound, we propose a new space of functions having bounded
\(\Pi\)-variation, which we call the curvewise \(\Pi\)-BV space and
we denote it by \({\rm BV}_\Pi^{\sf cw}(\X)\).
\begin{definition}[Curvewise \(\Pi\)-BV space]\label{def:BV_cw}
Let \((\X,\sfd,\mm)\) be a metric measure space. Let \(\Pi\) be a family
of \(\infty\)-test plans on \((\X,\sfd,\mm)\). Then we denote by
\({\rm BV}^{\sf cw}_\Pi(\X)\) the space of all those functions
\(f\in L^1(\mm)\) which admit a curvewise \(\Pi\)-bound \((g_n)_n\)
on \(\X\) with \(\sup_n\|g_n\|_{L^1(\mm)}<\infty\).
Given any \(f\in{\rm BV}^{\sf cw}_\Pi(\X)\) and \(\Omega\subseteq\X\)
open, we define
\begin{equation}\label{eq:def_curvewise_BV}
|{\boldsymbol D}f|^{\sf cw}_\Pi(\Omega)\coloneqq
\inf_{(g_n)_n}\limi_{n\to\infty}\int_\Omega g_n\,\d\mm,
\end{equation}
where the infimum is taken among all curvewise \(\Pi\)-bounds
\((g_n)_n\) for \(f\) on \(\Omega\). When \(\Pi\) is the totality
of \(\infty\)-test plans on \((\X,\sfd,\mm)\), we use the shorthand
notation \({\rm BV}^{\sf cw}(\X)\) and \(|{\boldsymbol D}f|^{\sf cw}\).
When \(\Pi=\{\ppi\}\) is a singleton, we use the shorthand notation
\({\rm BV}_\sppi^{\sf cw}(\X)\) and \(|{\boldsymbol D}f|_\sppi^{\sf cw}\).
\end{definition}
Observe that \({\rm BV}^{\sf cw}_\Pi(\X)\) can be equivalently
characterised as the set of all \(f\in L^1(\mm)\)
for which the quantity in the right-hand side of
\eqref{eq:def_curvewise_BV} (with \(\Omega=\X\)) is finite.
Moreover, given any
\(f\in{\rm BV}_\Pi^{\sf cw}(\X)\), we can extend the
function \(\Omega\mapsto|{\boldsymbol D}f|_\Pi^{\sf cw}(\Omega)\)
introduced in \eqref{eq:def_curvewise_BV}
to a set-function defined on all Borel sets via Carath\'{e}odory
construction, in the following way:
\begin{equation}\label{eq:ext_ptse_tv}
|{\boldsymbol D}f|_\Pi^{\sf cw}(B)\coloneqq\inf\big\{|{\boldsymbol D}f|
_\Pi^{\sf cw}(\Omega)\,:\,\Omega\subseteq\X\text{ open},\,B\subseteq
\Omega\big\},\quad\text{ for every }B\subseteq\X\text{ Borel.}
\end{equation}
We will not check whether the set-function
\(|{\boldsymbol D}f|_\Pi^{\sf cw}\) in \eqref{eq:ext_ptse_tv}
actually defines a Borel measure on \((\X,\sfd)\) when \(\Pi\) is an arbitrary family of
\(\infty\)-test plans. However, this is the case in the specific
situation where \(\Pi\) is a master family for \({\rm BV}(\X)\),
as granted by the following result.
\begin{theorem}\label{thm:equiv_BV_cw}
Let \((\X,\sfd,\mm)\) be a metric measure space. Let \(\Pi\) be a
family of \(\infty\)-test plans on \((\X,\sfd,\mm)\). Then \({\rm BV}(\X)
\subseteq{\rm BV}_\Pi^{\sf cw}(\X)\subseteq{\rm BV}_\Pi(\X)\) and
\begin{equation}\label{eq:equiv_BV_cw_claim}
|{\boldsymbol D}f|_\Pi(\Omega)\leq|{\boldsymbol D}f|_\Pi^{\sf cw}(\Omega)
\leq|{\boldsymbol D}f|(\Omega),\quad\text{ for every }f\in{\rm BV}(\X)
\text{ and }\Omega\subseteq\X\text{ open.}
\end{equation}
Moreover, if \(\Pi\) is a master family for \({\rm BV}(\X)\),
then \({\rm BV}_\Pi^{\sf cw}(\X)={\rm BV}(\X)\) and
\(|{\boldsymbol D}f|_\Pi^{\sf cw}=|{\boldsymbol D}f|\) holds
for every \(f\in{\rm BV}(\X)\), thus in particular
\(|{\boldsymbol D}f|_\Pi^{\sf cw}\) is a finite Borel measure on
\((\X,\sfd)\).
\end{theorem}
\begin{proof}
\ \\
{\color{blue}\textsc{Step 1.}} First of all, we aim to show that
\({\rm BV}(\X)\subseteq{\rm BV}_\Pi^{\sf cw}(\X)\) and
\begin{equation}\label{eq:equiv_BV_aux1}
|{\boldsymbol D}f|_\Pi^{\sf cw}(\Omega)\leq|{\boldsymbol D}f|(\Omega),
\quad\text{ for every }f\in{\rm BV}(\X)\text{ and }\Omega\subseteq\X
\text{ open.}
\end{equation}
To prove it, pick any sequence
\((f_n)_{n\in\N}\subseteq{\rm LIP}_{loc}(\Omega)\cap L^1(\mm|_\Omega)\)
such that \(f_n\to f\) in \(L^1(\mm|_\Omega)\) and
\(\int_\Omega\lip_a(f_n)\,\d\mm\to|{\boldsymbol D}f|(\Omega)\),
whose existence is granted by Theorem \ref{thm:equiv_BV}. Fix any
\(\ppi\in\Pi\). Choose Borel functions \(\bar f\colon\X\to\R\) and
\(\bar f_n\colon\X\to\R\), \(n\in\N\), having the following properties:
\(\bar f=f\) holds \(\mm\)-a.e.\ on \(\Omega\), \(\bar f_n=f_n\)
on \(\Omega\) for every \(n\in\N\), \(\bar f_n=\bar f\) on
\(\X\setminus\Omega\) for every \(n\in\N\), and
\(\int_{\X\setminus\Omega}|\bar f|\,\d\mm<+\infty\). In particular,
it holds that \(\bar f_n\to\bar f\) in \(L^1(\mm)\). Now observe that
\begin{equation}\label{eq:equiv_BV_aux2}
\bar f_n\circ\gamma\to\bar f\circ\gamma\;\text{ in }L^1(0,1),
\quad\text{ for }\ppi\text{-a.e.\ }\gamma,
\end{equation}
possibly after passing to a (not relabelled) subsequence in \(n\).
Indeed, we can estimate
\[\begin{split}
\int\!\!\!\int_0^1\big|(\bar f_n\circ\gamma)(t)-(\bar f\circ\gamma)(t)\big|
\,\d t\,\d\ppi(\gamma)&=\int\!\!\!\int_0^1|\bar f_n-\bar f|\circ\e_t\,\d t\,\d\ppi\\
&\leq{\rm Comp}(\ppi)\int|\bar f_n-\bar f|\,\d\mm\to 0,
\quad\text{ as }n\to\infty,
\end{split}\]
thus up to a not relabelled subsequence we have that
\(\int_0^1\big|(\bar f_n\circ\gamma)(t)-(\bar f\circ\gamma)(t)
\big|\,\d t\to 0\) as \(n\to\infty\) for \(\ppi\)-a.e.\ \(\gamma\),
so that accordingly the claimed property \eqref{eq:equiv_BV_aux2}
is proven.

Now pick a \(\ppi\)-null Borel set \(N\) of curves where the property
in \eqref{eq:equiv_BV_aux2} fails. Fix any \(\gamma\notin N\) and
\(0<a<b<1\) such that \(\gamma((a,b))\subseteq\Omega\). Thanks to the
lower semicontinuity and the locality of the total variation measures,
we obtain that
\[\begin{split}
|D(f\circ\gamma)|((a,b))&=|D(\bar f\circ\gamma)|((a,b))\leq
\limi_{n\to\infty}|D(\bar f_n\circ\gamma)|((a,b))=
\limi_{n\to\infty}\int_a^b|(\bar f_n\circ\gamma)'_t|\,\d t\\
&\leq\limi_{n\to\infty}\int_a^b\lip_a(\bar f_n)(\gamma_t)|\dot\gamma_t|\,\d t
\leq\limi_{n\to\infty}\int_a^b\big(\1_\Omega\,\lip_a(f_n)\big)(\gamma_t)|\dot\gamma_t|\,\d t.
\end{split}\]
This shows that \(\big(\1_\Omega\,\lip_a(f_n)\big)_n\)
is a curvewise \(\Pi\)-bound for \(f\) on \(\Omega\).
When considering \(\Omega=\X\), we get that
\(f\in{\rm BV}_\Pi^{\sf cw}(\X)\). The same estimates also give
\[
|{\boldsymbol D}f|_\Pi^{\sf cw}(\Omega)\leq\lim_{n\to\infty}
\int\1_\Omega\,\lip_a(f_n)\,\d\mm=|{\boldsymbol D}f|(\Omega),
\]
whence \eqref{eq:equiv_BV_aux1} follows.\\
{\color{blue}\textsc{Step 2.}} Next we claim that
\({\rm BV}_\Pi^{\sf cw}(\X)\subseteq{\rm BV}_\Pi(\X)\) and
\begin{equation}\label{eq:equiv_BV_aux3}
|{\boldsymbol D}f|_\Pi(\Omega)\leq|{\boldsymbol D}f|_\Pi^{\sf cw}(\Omega),
\quad\text{ for every }f\in{\rm BV}_\Pi^{\sf cw}(\X)\text{ and }
\Omega\subseteq\X\text{ open.}
\end{equation}
In order to prove it, let \(\ppi\in\Pi\) and \(\varepsilon>0\)
be fixed. We can find a curvewise \(\Pi\)-bound \((g_n)_n\) for
\(f\) on \(\Omega\) satisfying \(\limi_n\int_\Omega g_n\,\d\mm\leq
|{\boldsymbol D}f|_\Pi^{\sf cw}(\Omega)+\varepsilon\). Thanks to
Lemma \ref{lem:bdry_cond_BV}, there exists a \(\ppi\)-null Borel
set \(N\) of curves such that \(f\circ\gamma\in{\rm BV}(0,1)\) and
\(|D(f\circ\gamma)|(\{0,1\})=0\) for all \(\gamma\notin N\).
Given any \(\gamma\notin N\), we can find \(N_\gamma\in\N\cup\{\infty\}\)
and \(\{a_i^\gamma\}_{i<N_\gamma},\{b_i^\gamma\}_{i<N_\gamma}
\subseteq[0,1]\) such that
\begin{equation}\label{eq:equiv_BV_aux4}
a_i^\gamma<b_i^\gamma<a_{i+1}^\gamma<b_{i+1}^\gamma,\qquad
\gamma^{-1}(\Omega)\cap(0,1)=\bigcup_{i<N_\gamma}(a_i^\gamma,b_i^\gamma).
\end{equation}
Therefore, for any \(\gamma\notin N\) we may estimate
\[\begin{split}
\gamma_\#|D(f\circ\gamma)|(\Omega)&=
|D(f\circ\gamma)|\big(\gamma^{-1}(\Omega)\cap(0,1)\big)
\overset{\eqref{eq:equiv_BV_aux4}}=
\sum_{i<N_\gamma}|D(f\circ\gamma)|((a_i^\gamma,b_i^\gamma))\\
&\leq\sum_{i<N_\gamma}\limi_{n\to\infty}
\int_{a_i^\gamma}^{b_i^\gamma}g_n(\gamma_t)|\dot\gamma_t|\,\d t
\overset{\star}\leq\limi_{n\to\infty}\sum_{i<N_\gamma}
\int_{a_i^\gamma}^{b_i^\gamma}g_n(\gamma_t)|\dot\gamma_t|\,\d t\\
&=\limi_{n\to\infty}\int_{\gamma^{-1}(\Omega)}g_n(\gamma_t)|\dot\gamma_t|\,\d t,
\end{split}\]
where the starred inequality is granted by Fatou's lemma.
By integrating over \(\ppi\), we obtain
\[\begin{split}
\int\gamma_\#|D(f\circ\gamma)|(\Omega)\,\d\ppi(\gamma)
&\leq\int\limi_{n\to\infty}
\int_{\gamma^{-1}(\Omega)}g_n(\gamma_t)|\dot\gamma_t|\,\d t
\,\d\ppi(\gamma)\\
&\overset{\star}\leq\limi_{n\to\infty}\int\!\!\!\int_
{\gamma^{-1}(\Omega)}g_n(\gamma_t)|\dot\gamma_t|\,\d t\,\d\ppi(\gamma)\\
&\leq{\rm Lip}(\ppi)\limi_{n\to\infty}\int\!\!\!\int
(\1_\Omega\,g_n)\circ\e_t\,\d t\,\d\ppi\\
&\leq{\rm Comp}(\ppi){\rm Lip}(\ppi)\limi_{n\to\infty}\int_\Omega g_n\,\d\mm\\
&\leq{\rm Comp}(\ppi){\rm Lip}(\ppi)
\big(|{\boldsymbol D}f|_\Pi^{\sf cw}(\Omega)+\varepsilon\big),
\end{split}\]
where the starred inequality is a consequence of Fatou's lemma.
By letting \(\varepsilon\searrow 0\), and exploiting the arbitrariness
of \(\ppi\in\Pi\) and the open set \(\Omega\), we conclude that
\(f\in{\rm BV}_\Pi(\X)\) and that \(|{\boldsymbol D}f|_\Pi(\Omega)\leq
|{\boldsymbol D}f|_\Pi^{\sf cw}(\Omega)\) for every \(\Omega\subseteq\X\)
open, thus proving the claimed property \eqref{eq:equiv_BV_aux3}.\\
{\color{blue}\textsc{Step 3.}} What is left to prove is only
the last part of the statement. If \(\Pi\) is a master family
for \({\rm BV}(\X)\), then \eqref{eq:equiv_BV_cw_claim} forces
the identities \({\rm BV}_\Pi^{\sf cw}(\X)={\rm BV}(\X)\), and
\(|{\boldsymbol D}f|_\Pi^{\sf cw}(\Omega)=|{\boldsymbol D}f|(\Omega)\)
for every \(f\in{\rm BV}(\X)\) and \(\Omega\subseteq\X\) open.
By virtue of the definition \eqref{eq:ext_ptse_tv} and the outer
regularity of the measure \(|{\boldsymbol D}f|\), we deduce that
\(|{\boldsymbol D}f|_\Pi^{\sf cw}(B)=|{\boldsymbol D}f|(B)\) holds
for every Borel set \(B\subseteq\X\), thus in particular
\(|{\boldsymbol D}f|_\Pi^{\sf cw}\) is a finite Borel measure as well.
The statement is achieved.
\end{proof}
For the sake of clarity, we report the following immediate
consequence of Theorem \ref{thm:equiv_BV_cw}.
\begin{corollary}\label{cor:equiv_BV_cw}
Let \((\X,\sfd,\mm)\) be a metric measure space. Then
\({\rm BV}^{\sf cw}(\X)={\rm BV}(\X)\) and
\[
|{\boldsymbol D}f|^{\sf cw}=|{\boldsymbol D}f|,
\quad\text{ for every }f\in{\rm BV}(\X).
\]
\end{corollary}
\begin{proof}
Since the totality \(\Pi_\infty(\X)\) of \(\infty\)-test plans
on \((\X,\sfd,\mm)\) is a master family for \({\rm BV}(\X)\),
the claim directly follows from the last part of the statement
of Theorem \ref{thm:equiv_BV_cw}.
\end{proof}
\subsection{A master curvewise plan for BV}\label{ss:master_cw_tp}
Aim of this section is to build a \(\infty\)-test plan \(\ppi_{\sf m}\)
(concentrated on geodesics when the underlying space is non-branching
\(\sf CD\)) such that
\[
{\rm BV}^{\sf cw}_{\sppi_{\sf m}}(\X)={\rm BV}(\X).
\]
The test plan \(\ppi_{\sf m}\) will be obtained by suitably combining the
elements that constitute the master sequence for \({\rm BV}(\X)\)
provided by Theorem \ref{thm:countable_master_tp}.
Before passing to the construction of the plan \(\ppi_{\sf m}\)
in Theorem \ref{thm:master_cw_plan}, we need to prove the following
technical lemma, concerning the behaviour of
\({\rm BV}_\Pi^{\sf cw}(\X)\) under different manipulations of the
family \(\Pi\) of \(\infty\)-test plans.
\begin{lemma}\label{lem:basic_prop_cw}
Let \((\X,\sfd,\mm)\) be a metric measure space. Then the following
properties hold:
\begin{itemize}
\item[\(\rm i)\)] Let \((\ppi_n)_{n\in\N}\subseteq\Pi_\infty(\X)\)
be given. Let \((\alpha_n)_{n\in\N}\subseteq(0,1)\) be a sequence
with \(\sum_{n=1}^\infty\alpha_n=1\). Suppose \(\sup_n\Lip(\ppi_n)<
+\infty\) and \(\sum_{n=1}^\infty\alpha_n{\rm Comp}(\ppi_n)<+\infty\).
Then \(\ppi\coloneqq\sum_{n=1}^\infty\alpha_n\ppi_n\) is a
\(\infty\)-test plan satisfying
\({\rm BV}_\sppi^{\sf cw}(\X)={\rm BV}_{\{\sppi_n\}_n}^{\sf cw}(\X)\) and
\[
|{\boldsymbol D}f|_\sppi^{\sf cw}=|{\boldsymbol D}f|_{\{\sppi_n\}_n}
^{\sf cw},\quad\text{ for every }f\in{\rm BV}_{\{\sppi_n\}_n}^{\sf cw}(\X).
\]
\item[\(\rm ii)\)] Let \(\Pi=\{\ppi_\lambda\}_{\lambda\in\Lambda}\)
be an arbitrary family of \(\infty\)-test plans on \((\X,\sfd,\mm)\).
For any \(\lambda\in\Lambda\), fix a number \(k(\lambda)\in\N\) and a
subdivision \(0=t_\lambda^0<t_\lambda^1\ldots<t_\lambda^{k(\lambda)}=1\)
of \([0,1]\). Define
\[
\ppi_\lambda^i\coloneqq\big({\rm restr}_{t_\lambda^{i-1}}^{t_\lambda^i
}\big)_\#\ppi_\lambda\in\Pi_\infty(\X),\quad\text{ for every }
\lambda\in\Lambda\text{ and }i=1,\ldots,k(\lambda).
\]
Then the family \(\hat\Pi\coloneqq\big\{\ppi_\lambda^i\,:\,
\lambda\in\Lambda,\,1\leq i\leq k(\lambda)\big\}\) satisfies
\({\rm BV}_{\hat\Pi}^{\sf cw}(\X)={\rm BV}_\Pi^{\sf cw}(\X)\) and
\[
|{\boldsymbol D}f|_{\hat\Pi}^{\sf cw}=|{\boldsymbol D}f|_\Pi^{\sf cw},
\quad\text{ for every }f\in{\rm BV}_\Pi^{\sf cw}(\X).
\]
\end{itemize}
\end{lemma}
\begin{proof}
\ \\
{\color{blue}i)} Observe that a Borel set \(N\subseteq C([0,1],\X)\)
is \(\ppi\)-null if and only if it is \(\ppi_n\)-null for all \(n\in\N\).
This implies that curvewise \(\ppi\)-bounds and curvewise
\(\{\ppi_n\}_n\)-bounds coincide, yielding i).\\
{\color{blue}ii)} Fix any \(f\in L^1(\mm)\) and an open set
\(\Omega\subseteq\X\). It is then sufficient to show that a given
sequence \((g_n)_{n\in\N}\subseteq L^1(\mm|_\Omega)\) is a curvewise
\(\Pi\)-bound for \(f\) on \(\Omega\) if and only if it is a curvewise
\(\hat\Pi\)-bound for \(f\) on \(\Omega\). Suppose the former holds.
Fix any \(\lambda\in\Lambda\) and \(i=1,\ldots,k(\lambda)\). Let \(N\)
be a \(\ppi_\lambda\)-null Borel set of curves such that the property
in \eqref{eq:def_BV_bound} holds for every fixed \(\gamma\notin N\);
here and in the rest of the proof, we are assuming to have
fixed a Borel representative of \(f\).
In particular, \(\sigma\coloneqq{\rm restr}_{t_\lambda^{i-1}}
^{t_\lambda^i}(\gamma)\) has this property: if
\(0<a<b<1\) and \(\sigma((a,b))\subseteq\Omega\), then
\[
|D(f\circ\sigma)|((a,b))=|D(f\circ\gamma)|((a',b'))\leq
\limi_{n\to\infty}\int_{a'}^{b'}g_n(\gamma_t)|\dot\gamma_t|\,\d t
=\limi_{n\to\infty}\int_a^b g_n(\sigma_t)|\dot\sigma_t|\,\d t,
\]
where we set \(a'\coloneqq(1-a)t_\lambda^{i-1}+a\,t_\lambda^i\)
and \(b'\coloneqq(1-b)t_\lambda^{i-1}+b\,t_\lambda^i\). Hence, given
that the set of such curves \(\sigma\)'s is \(\ppi_\lambda^i\)-null,
we have proved that \((g_n)_n\) is a curvewise \(\hat\Pi\)-bound for
\(f\) on \(\Omega\).

Conversely, suppose \((g_n)_n\) is a curvewise \(\hat\Pi\)-bound for
\(f\) on \(\Omega\). Fix any \(\lambda\in\Lambda\). Given any index
\(i=1,\ldots,k(\lambda)\), we can find a \(\ppi_\lambda^i\)-null Borel
set \(N_i\) of curves such that \eqref{eq:def_BV_bound} holds for
all \(\sigma\notin N_i\). Thanks to Lemma \ref{lem:bdry_cond_BV},
we can find a \(\ppi_\lambda\)-null Borel set \(\tilde N\) of curves
such that
\begin{equation}\label{eq:techn_curvewise_aux}
|D(f\circ\gamma)|\big(\{t_\lambda^1,\ldots,t_\lambda^{k(\lambda)-1}\}\big)
=0,
\quad\text{ for every }\gamma\notin\tilde N.
\end{equation}
Now let us consider the \(\ppi_\lambda\)-null set \(N\) of curves,
which is defined as
\[
N\coloneqq\tilde N\cup\bigcup_{i=1}^{k(\lambda)}
\big({\rm restr}_{t_\lambda^{i-1}}^{t_\lambda^i}\big)^{-1}(N_i).
\]
Fix \(\gamma\notin N\) and \(0<a<b<1\) with \(\gamma((a,b))\subseteq\Omega\).
For any \(i=1,\ldots,k(\lambda)\), we denote by \(I_i\) the open
interval \((a,b)\cap(t_\lambda^{i-1},t_\lambda^i)\). Given that
\(\sigma^i\coloneqq{\rm restr}_{t_\lambda^{i-1}}^{t_\lambda^i}(\gamma)
\notin N_i\), we may estimate
\[\begin{split}
|D(f\circ\gamma)|((a,b))&\overset{\eqref{eq:techn_curvewise_aux}}=
\sum_{i=1}^{k(\lambda)}|D(f\circ\gamma)|(I_i)=
\sum_{i=1}^{k(\lambda)}|D(f\circ\sigma^i)|(I'_i)
\leq\sum_{i=1}^{k(\lambda)}\limi_{n\to\infty}
\int_{I'_i}g_n(\sigma^i_t)|\dot\sigma^i_t|\,\d t\\
&\overset{\phantom{\eqref{eq:techn_curvewise_aux}}}\leq
\limi_{n\to\infty}\sum_{i=1}^{k(\lambda)}
\int_{I'_i}g_n(\sigma^i_t)|\dot\sigma^i_t|\,\d t
=\limi_{n\to\infty}\int_a^b g_n(\gamma_t)|\dot\gamma_t|\,\d t,
\end{split}\]
where we set
\(I'_i\coloneqq{\rm restr}_{t_\lambda^{i-1}}^{t_\lambda^i}(I_i)\).
This shows that \((g_n)_n\) is a curvewise \(\Pi\)-bound for \(f\)
on \(\Omega\).
\end{proof}
We are now in a position to build the master curvewise plan
\(\ppi_{\sf m}\) for BV.
\begin{theorem}[Master curvewise plan for BV]\label{thm:master_cw_plan}
Let \((\X,\sfd,\mm)\) be a metric measure space. Then there exists
a \(\infty\)-test plan \(\ppi_{\sf m}\) on \((\X,\sfd,\mm)\) such that
\({\rm BV}_{\sppi_{\sf m}}^{\sf cw}(\X)={\rm BV}(\X)\) and
\[
|{\boldsymbol D}f|_{\sppi_{\sf m}}^{\sf cw}=|{\boldsymbol D}f|,
\quad\text{ for every }f\in{\rm BV}(\X).
\]
Moreover, if \((\X,\sfd,\mm)\) is a non-branching \({\sf CD}(K,N)\)
space for some \(K\in\R\), \(N\in(1,\infty)\), and the measure \(\mm\)
is finite, then we can additionally require that \(\ppi_{\sf m}\) is
concentrated on geodesics.
\end{theorem}
\begin{proof}
First, let \(\Pi=\{\ppi_i\}_{i\in\N}\) be a master family
for \({\rm BV}(\X)\), whose existence is granted by Theorem
\ref{thm:countable_master_tp}; the same result ensures
that, in the non-branching \({\sf CD}(K,N)\) case, each
\(\ppi_i\) can be additionally chosen to be concentrated on geodesics.
Given any \(i\in\N\), pick \(n_i\in\N\) such that
\({\rm Lip}(\ppi_i)\leq n_i\) and define the \(\infty\)-test
plans \((\ppi_i^j)_{j=1}^{n_i}\) as
\[
\ppi_i^j\coloneqq\big({\rm restr}_{(j-1)/n_i}^{j/n_i}\big)_\#\ppi_i,
\quad\text{ for every }j=1,\ldots,n_i.
\]
Notice that \({\rm Lip}(\ppi_i^j)\leq 1\). Now define
\(\hat\Pi\coloneqq\big\{\ppi_i^j\,:\,i\in\N,\,j=1,\ldots,n_i\big\}\).
Theorem \ref{thm:equiv_BV_cw} and item ii) of Lemma
\ref{lem:basic_prop_cw} ensure that
\({\rm BV}(\X)={\rm BV}_{\hat\Pi}^{\sf cw}(\X)\) and
\(|{\boldsymbol D}f|=|{\boldsymbol D}f|_{\hat\Pi}^{\sf cw}\)
for all \(f\in{\rm BV}(\X)\). Moreover, let us relabel \(\hat\Pi\) as
\(\{\ppi^n\}_{n\in\N}\). For any \(n\in\N\), we define
\(\alpha_n\in(0,1)\) as
\[
\alpha_n\coloneqq\frac{1}{2^n\alpha\max\{{\rm Comp}(\ppi^n),1\}},
\quad\text{ where }\alpha\coloneqq\sum_{m\in\N}
\frac{1}{2^m\max\{{\rm Comp}(\ppi^m),1\}}.
\]
Finally, we define the \(\infty\)-test plan
\(\ppi_{\sf m}\in\Pi_\infty(\X)\) as
\(\ppi_{\sf m}\coloneqq\sum_{n\in\N}\alpha_n\ppi^n\).
Observe that in the non-branching \({\sf CD}(K,N)\) case
we have that \(\ppi_{\sf m}\) is concentrated on \({\rm Geo}(\X)\).
What previously proved and item i) of Lemma \ref{lem:basic_prop_cw}
imply that \({\rm BV}(\X)={\rm BV}_{\sppi_{\sf m}}^{\sf cw}(\X)\) and
that \(|{\boldsymbol D}f|=|{\boldsymbol D}f|_{\sppi_{\sf m}}^{\sf cw}\)
for every \(f\in{\rm BV}(\X)\), thus yielding the statement.
\end{proof}
\appendix
\section{Comparison with the AM-BV space}\label{app:AM-BV}
Let \((\X,\sfd,\mm)\) be a metric measure space and
\(\Omega\subseteq\X\) an open set. We denote by
\(\mathcal C(\Omega)\) the family of all non-constant,
rectifiable curves \(\gamma\colon I\to\Omega\), where
\(I\subseteq\R\) is a compact interval. Given any
\(\gamma\in\mathcal C(\Omega)\), we will denote by \(I_\gamma\)
its domain of definition. For any non-negative Borel function
\(G\colon\Omega\to[0,+\infty]\), we denote the line integral
of \(G\) along \(\gamma\) by
\[
\int_\gamma G\coloneqq\int_{\min I_\gamma}^{\max I_\gamma}
G(\gamma_t)|\dot\gamma_t|\,\d t.
\]
\begin{definition}[Approximation modulus \cite{Martio16}]
Let \((\X,\sfd,\mm)\) be a metric measure space, \(\Omega\subseteq\X\)
an open set. Let \(\Gamma\subseteq\mathcal C(\Omega)\) be a given
family of curves. Then a sequence \((\rho_i)_{i\in\N}\) of non-negative
Borel functions \(\rho_i\colon\Omega\to\R\) is said to be
\(AM_\Omega\)-admissible for \(\Gamma\) provided it holds
\[
\limi_{i\to\infty}\int_\gamma\rho_i\geq 1,
\quad\text{ for every }\gamma\in\Gamma.
\]
The approximation modulus of \(\Gamma\) in \(\Omega\) is defined as
\[
AM_\Omega(\Gamma)\coloneqq\inf_{(\rho_i)_i}\limi_{i\to\infty}
\int_\Omega\rho_i\,\d\mm,
\]
where the infimum is taken among all \(AM_\Omega\)-admissible
sequences \((\rho_i)_i\) for \(\Gamma\). Moreover, a property
\(\mathcal P=\mathcal P(\gamma)\) is said to hold for
\(AM_\Omega\)-a.e.\ curve \(\gamma\) provided there exists a
family of curves \(\Gamma_0\subseteq\mathcal C(\Omega)\) with
\(AM_\Omega(\Gamma_0)=0\) such that \(\mathcal P(\gamma)\) holds
for every \(\gamma\in\mathcal C(\Omega)\setminus\Gamma_0\).
\end{definition}

It holds that \(AM_\Omega\) is an outer measure on
\(\mathcal C(\Omega)\). When \(\Omega=\X\), we just write \(AM\)
in place of \(AM_\X\). Observe that \(\mathcal C(\Omega)\subseteq
\mathcal C(\X)\) and that \(AM(\Gamma)\leq AM_\Omega(\Gamma)\)
for every \(\Gamma\subseteq\mathcal C(\Omega)\).
\begin{lemma}\label{lem:null_AM_vs_tp}
Let \((\X,\sfd,\mm)\) be a metric measure space. Let
\(\Gamma\subseteq C([0,1],\X)\) be a family of curves
such that \(AM(\Gamma)=0\). Fix a \(\infty\)-test plan
\(\ppi\) on \((\X,\sfd,\mm)\). Then there exists a Borel set
\(\Gamma_0\subseteq C([0,1],\X)\) such that \(\Gamma\subseteq\Gamma_0\)
and \(\ppi(\Gamma_0)=0\).
\end{lemma}
\begin{proof}
Given \(n\in\N\), pick an AM-admissible
sequence \((\rho_i^n)_i\) for \(\Gamma\) such that
\(\limi_i\int\rho_i^n\,\d\mm\leq 1/n\). Let us consider the Borel sets
\begin{equation}\label{eq:def_Gamma_n}
\Gamma_n\coloneqq\bigg\{\gamma\in C([0,1],\X)\;\bigg|\;
\limi_{i\to\infty}\int_\gamma\rho_i^n\geq 1\bigg\},
\quad\text{ for every }n\in\N.
\end{equation}
Then the set \(\Gamma_0\coloneqq\bigcap_n\Gamma_n\) is Borel and
contains \(\Gamma\). We aim to show that \(\ppi(\Gamma_0)=0\). Since
\[\begin{split}
\ppi(\Gamma_0)&\leq\ppi(\Gamma_n)=\int\1_{\Gamma_n}\,\d\ppi
\overset{\eqref{eq:def_Gamma_n}}\leq
\int\bigg(\limi_{i\to\infty}\int_\gamma\rho_i^n\bigg)\d\ppi(\gamma)
\overset{\star}\leq\limi_{i\to\infty}\int\!\!\!
\int_0^1\rho_i^n(\gamma_t)|\dot\gamma_t|\,\d t\,\d\ppi(\gamma)\\
&\leq{\rm Lip}(\ppi)\limi_{i\to\infty}\int_0^1\!\!\!\int\rho_i^n\circ
\e_t\,\d\ppi\,\d t\leq{\rm Comp}(\ppi){\rm Lip}(\ppi)
\limi_{i\to\infty}\int\rho_i^n\,\d\mm\leq
\frac{{\rm Comp}(\ppi){\rm Lip}(\ppi)}{n}
\end{split}\]
holds for every \(n\in\N\) (where the starred inequality can be
justified by using Fatou's lemma), by letting \(n\to\infty\) we
conclude that \(\ppi(\Gamma_0)=0\), as desired.
\end{proof}

The following remark is an easy consequence of the
definition of approximation modulus:
\begin{remark}\label{rmk:behaviour_AM_restr}{\rm
Let \(\Gamma,\Gamma'\subseteq\mathcal C(\X)\) be two families of
curves having the following property: given any \(\gamma\in\Gamma\),
some subcurve \(\sigma\) of \(\gamma\) belongs to \(\Gamma'\),
meaning that there exists a compact subinterval \(I\) of \(I_\gamma\)
such that \(\sigma\coloneqq\gamma|_I\in\Gamma'\). Then it holds
that \(AM(\Gamma)\leq AM(\Gamma')\).
\fr}\end{remark}

We denote by \(\mathcal L^1(\mm)\) the family of all Borel functions
\(f\colon\X\to\R\) such that \(\int|f|\,\d\mm<+\infty\).
In particular, the Lebesgue space \(L^1(\mm)\) is the quotient of
\(\mathcal L^1(\mm)\) obtained by identifying those Borel functions
which agree up to \(\mm\)-negligible sets.
\begin{definition}[\(BV_{AM}\) upper bound \cite{Martio16-2}]
Let \((\X,\sfd,\mm)\) be a metric measure space, \(\Omega\subseteq\X\)
an open set. Let \(f\in\mathcal L^1(\mm|_\Omega)\) be given.
Then we say that a sequence \((g_i)_{i\in\N}\) of non-negative Borel
functions \(g_i\colon\Omega\to\R\) is a \(BV_{AM}\) upper bound for
\(f\) on \(\Omega\) provided it holds that
\begin{equation}\label{eq:def_AM_upper_bound}
|D(f\circ\gamma)|(I_\gamma)\leq\limi_{i\to\infty}\int_\gamma g_i,
\quad\text{ for }AM_\Omega\text{-a.e.\ }\gamma.
\end{equation}
\end{definition}

Being both sides of \eqref{eq:def_AM_upper_bound} invariant
under reparametrisations of \(\gamma\), we have that \((g_i)_i\)
is a \(BV_{AM}\) upper bound for \(f\) on \(\Omega\) if and only if
\eqref{eq:def_AM_upper_bound} holds for \(AM_\Omega\)-a.e.\ \(\gamma\)
having constant speed. Moreover, as proven in
\cite[Lemma 2.2]{Martio16-2}, we have that \((g_i)_i\)
is a \(BV_{AM}\) upper bound for \(f\) on \(\Omega\)
if and only if for \(AM_\Omega\)-a.e.\ \(\gamma\) it holds
\begin{equation}\label{eq:def_AM_upper_bound_alt}
|D(f\circ\gamma)|([a,b])\leq\limi_{i\to\infty}\int_{\gamma|_{[a,b]}}g_i,
\quad\text{ for every }a,b\in I_\gamma\text{ with }a<b.
\end{equation}
\begin{definition}[AM-BV space \cite{Martio16-2}]
Let \((\X,\sfd,\mm)\) be a metric measure space.
Fix a function \(f\in L^1(\mm)\). Then we say that \(f\) belongs to
the space \({\rm BV}_{AM}(\X)\) provided there exist a
representative \(\bar f\in\mathcal L^1(\mm)\) of \(f\) and a
\(BV_{AM}\) upper bound \((g_i)_i\) for \(\bar f\) such that
\[
\limi_{i\to\infty}\int g_i\,\d\mm<+\infty.
\]
Given any \(f\in{\rm BV}_{AM}(\X)\) and \(\Omega\subseteq\X\) open,
we define
\begin{equation}\label{eq:def_tvm_AM_on_open}
|{\boldsymbol D}f|_{AM}(\Omega)\coloneqq\inf_{\bar f,(g_i)_i}
\limi_{i\to\infty}\int_\Omega g_i\,\d\mm,
\end{equation}
where the infimum is taken among all representatives
\(\bar f\in\mathcal L^1(\mm)\) of \(f\) and all \(BV_{AM}\)
upper bounds \((g_i)_i\) for \(\bar f\) on \(\Omega\).
\end{definition}

As usual, the set-function \(|{\boldsymbol D}f|_{AM}\) defined in
\eqref{eq:def_tvm_AM_on_open} can be extended to all Borel sets via
Carath\'{e}odory construction, as follows:
\begin{equation}\label{eq:def_tvm_AM}
|{\boldsymbol D}f|_{AM}(B)\coloneqq
\inf\big\{|{\boldsymbol D}f|_{AM}(\Omega)\,:\,
\Omega\subseteq\X\text{ open},\,B\subseteq\Omega\big\},
\quad\text{ for every }B\subseteq\X\text{ Borel.}
\end{equation}
As proven in \cite{Martio16-2}, it holds that \(|{\boldsymbol D}f|_{AM}\)
as in \eqref{eq:def_tvm_AM} is a finite Borel measure on \((\X,\sfd)\).
\begin{theorem}[\({\rm BV}_{AM}(\X)={\rm BV}(\X)\)]
\label{thm:equiv_BV_AM}
Let \((\X,\sfd,\mm)\) be a metric measure space. Then it holds
\[
{\rm BV}_{AM}(\X)={\rm BV}(\X),\qquad|{\boldsymbol D}f|_{AM}=
|{\boldsymbol D}f|\;\text{ for every }f\in{\rm BV}(\X).
\]
\end{theorem}
\begin{proof}
By virtue of Theorem \ref{thm:equiv_BV_cw}, it suffices to show
\({\rm BV}(\X)\subseteq{\rm BV}_{AM}(\X)\subseteq{\rm BV}^{\sf cw}(\X)\) and
\[
|{\boldsymbol D}f|^{\sf cw}(\Omega)\leq|{\boldsymbol D}f|_{AM}(\Omega)
\leq|{\boldsymbol D}f|(\Omega),\quad\text{ for every }f\in{\rm BV}(\X)
\text{ and }\Omega\subseteq\X\text{ open.}
\]
We will prove it in two steps.\\
{\color{blue}\textsc{Step 1.}} First, we want to prove that
\({\rm BV}(\X)\subseteq{\rm BV}_{AM}(\X)\), and that
\(|{\boldsymbol D}f|_{AM}(\Omega)\leq|{\boldsymbol D}f|(\Omega)\)
for every \(f\in{\rm BV}(\X)\) and \(\Omega\subseteq\X\) open.
Thanks to Theorem \ref{thm:equiv_BV}, we can find a sequence
\((f_i)_i\subseteq\LIP_{loc}(\Omega)\cap L^1(\mm|_\Omega)\)
such that \(f_i\to f\) in \(L^1(\mm|_\Omega)\) and
\(\int_\Omega\lip_a(f_i)\,\d\mm\to|{\boldsymbol D}f|(\Omega)\).
Fix a representative \(\bar f\in\mathcal L^1(\mm|_\Omega)\) of \(f\).
It follows from Fuglede's lemma \cite[Lemma 2.1]{Bjorn-Bjorn11} that
(up to a not relabelled subsequence) it holds that
\(\int_\gamma|f_i-\bar f|\to 0\) as \(i\to\infty\)
for \(AM_\Omega\)-a.e.\ \(\gamma\). In particular, we have that
\(f_i\circ\gamma\to\bar f\circ\gamma\) strongly in \(L^1(0,1)\)
for \(AM_\Omega\)-a.e.\ \(\gamma\) having constant speed.
By using the lower semicontinuity of the total variation measures,
we thus obtain that
\[\begin{split}
|D(\bar f\circ\gamma)|(I_\gamma)&\leq
\limi_{i\to\infty}|D(f_i\circ\gamma)|(I_\gamma)
=\limi_{i\to\infty}\int_{\min I_\gamma}^{\max I_\gamma}
|(f_i\circ\gamma)'_t|\,\d t\\
&\leq\limi_{i\to\infty}\int_{\min I_\gamma}^{\max I_\gamma}
\lip_a(f_i)(\gamma_t)|\dot\gamma_t|\,\d t
=\limi_{i\to\infty}\int_\gamma\lip_a(f_i).
\end{split}\]
This shows that \(\big(\lip_a(f_i)\big)_i\) is a \(BV_{AM}\)
upper bound for \(\bar f\) on \(\Omega\). Therefore,
we conclude that
\(|{\boldsymbol D}f|_{AM}(\Omega)\leq\limi_i\int_\Omega
\lip_a(f_i)\,\d\mm=|{\boldsymbol D}f|(\Omega)\), as desired.\\
{\color{blue}\textsc{Step 2.}} Next we aim to prove that
\({\rm BV}_{AM}(\X)\subseteq{\rm BV}^{\sf cw}(\X)\), and that
\(|{\boldsymbol D}f|^{\sf cw}(\Omega)\leq|{\boldsymbol D}f|_{AM}(\Omega)\)
for every \(f\in{\rm BV}_{AM}(\X)\) and \(\Omega\subseteq\X\) open.
Given any \(\varepsilon>0\), pick a representative
\(\bar f\in\mathcal L^1(\mm)\) of \(f\) and a \(BV_{AM}\)
upper bound \((g_i)_i\) for \(\bar f\) on \(\Omega\) such that
\(\limi_i\int_\Omega g_i\,\d\mm\leq|{\boldsymbol D}f|_{AM}(\Omega)
+\varepsilon\). Fix a family \(\Gamma\subseteq\mathcal C(\Omega)\)
such that \(AM(\Gamma)=0\) and \eqref{eq:def_AM_upper_bound_alt}
holds for all \(\gamma\notin\Gamma\). In light of Remark
\ref{rmk:behaviour_AM_restr}, we can also assume without loss of
generality that if \(\sigma\in\Gamma\), then any curve
\(\gamma\in\mathcal C(\X)\) having \(\sigma\) as a subcurve
belongs to \(\Gamma\). Now let \(\ppi\) be a given \(\infty\)-test
plan on \((\X,\sfd,\mm)\). By applying Lemma \ref{lem:null_AM_vs_tp},
we can find a Borel set \(\Gamma_0\subseteq C([0,1],\X)\) such that
\(\Gamma\cap C([0,1],\X)\subseteq\Gamma_0\) and \(\ppi(\Gamma_0)=0\).
Now fix \(\gamma\in\LIP([0,1],\X)\setminus\Gamma_0\) and \(0<a<b<1\)
with \(\gamma((a,b))\subseteq\Omega\). Given that \(\gamma\notin\Gamma\),
we have that \(\gamma|_{[a,b]}\notin\Gamma\) as well, thus accordingly
\[
|D(\bar f\circ\gamma)|((a,b))\leq|D(\bar f\circ\gamma)|([a,b])
\overset{\eqref{eq:def_AM_upper_bound_alt}}\leq
\limi_{i\to\infty}\int_{\gamma|_{[a,b]}}g_i=
\limi_{i\to\infty}\int_a^b g_i(\gamma_t)|\dot\gamma_t|\,\d t.
\]
This shows that \((g_i)_i\) is a curvewise bound for \(f\)
on \(\Omega\). In particular, we deduce that
\[
|{\boldsymbol D}f|^{\sf cw}(\Omega)\leq
\limi_{i\to\infty}\int_\Omega g_i\,\d\mm
\leq|{\boldsymbol D}f|_{AM}(\Omega)+\varepsilon,
\]
whence by letting \(\varepsilon\searrow 0\) we conclude that
\(|{\boldsymbol D}f|^{\sf cw}(\Omega)\leq|{\boldsymbol D}f|_{AM}(\Omega)\), as desired.
\end{proof}
\section{Master test plan for \texorpdfstring{\(W^{1,1}\)}{W11}
on \texorpdfstring{\(\sf RCD\)}{RCD} spaces}\label{app:master_W11}
In \cite{DiMarinoPhD} many definitions of \(1\)-Sobolev spaces are presented and inclusions between them are discussed. Here, we consider a notion of space $W^{1,1}(\X)$ defined in duality with $\infty$-test plans (that, in \cite{DiMarinoPhD}, is denoted by $w-W^{1,1}(\X)$).
\begin{definition}[The space $W^{1,1}(\X)$]
Let $(\X,\sfd,\mm)$ be a metric measure space. We say that $f \in W^{1,1}(\X)$, provided $f \in L^1(\mm)$ and there exists $G \in L^1(\mm)$ non-negative, called $1$-weak upper gradient of $f$, so that
\[ |f(\gamma_1)-f(\gamma_0)| \le \int G(\gamma_t)|\dot \gamma_t|\, \d t, \qquad \ppi\text{-a.e.\ }\gamma,\]
for every $\infty$-test plan $\ppi$.

The $\mm$-a.e.\ minimal $G$ satisfying the above, denoted $|Df|_1$, is called minimal $1$-weak upper gradient.
\end{definition}
Notice that the well-posedness of the above definition follows from standard considerations as in Remark \ref{rem:BVwellposed}.
We claim now that $W^{1,1}(\X) \subseteq {\rm BV}(\X)$. Fix any $f \in W^{1,1}(\X)$. Given an arbitrary $\infty$-test plan $\ppi$, it is standard to see that for $\ppi$-a.e.\ $\gamma$ we have $f\circ \gamma \in W^{1,1}(0,1)$ and $(f\circ\gamma)'_t \le |Df|_1(\gamma_t)|\dot \gamma_t|$ for a.e.\ $t \in [0,1]$ (note, \emph{e.g.}, in \cite[Section 4.6]{DiMarinoPhD} the inclusion with the Beppo-Levi space $W_{BL}^{1,1}$). Moreover, for every $B\subseteq\X$ Borel and every $\infty$-test plan $\ppi$, we can therefore estimate
\[ 
\begin{split}
\int \gamma_\#|D(f\circ\gamma)|(B)\, \d \ppi &=\int\!\!\!\int_0^1  \1_{\gamma^{-1}(B)}(t)(f\circ\gamma)'(t)\, \d t\,\d \ppi \\
&\le \Lip(\ppi)\int\!\!\!\int_0^1 (\1_{B} |Df|_1)\circ \e_t\, \d t\,\d \ppi \\
&\le \Lip(\ppi)\int_B|Df|_1\, \d \mm.
\end{split}
\]
All in all, the above shows at the same time that $f \in {\rm BV}(\X)$ and $|{\boldsymbol D} f|\le |Df|_1\mm$.

Unfortunately, it is not always true that, if $f \in{\rm BV}(\X)$ with $|{\boldsymbol D} f| \ll \mm$, then $f$ belongs to $W^{1,1}(\X)$ and $\frac{\d |{\boldsymbol D} f|}{\d \mm}$ is a $1$-weak upper gradient. The reason being (see the discussion at the beginning of Section 4.6 and Example 4.5.4 in \cite{DiMarinoPhD}), that the BV-condition requires $f \circ \gamma$ to be only ${\rm BV}(0,1)$ along a.e.\ curve, while the $W^{1,1}$-condition  requires the composition $f \circ \gamma$ to be absolutely continuous. This discrepancy allows in general for the existence of counterexamples. Nevertheless, as proven in \cite{GigliHan14}, this is not the case in the ${\sf RCD}(K,N)$ setting \cite{Gigli12}. Given that the content of the current Appendix is used nowhere in the rest of the paper, we shall not provide the reader with the exact definition of the ${\sf RCD}(K,N)$-condition and refer to the references given in the Introduction. Here we will just use that they are also ${\sf CD}(K,N)$ spaces and that, thanks to \cite[Remark 3.5]{GigliHan14}, on ${\sf RCD}(K,N)$-spaces, for some $K \in \R,N\in [1,\infty)$, it holds that
\begin{equation}
f  \in{\rm BV}(\X)\text{ with } |{\boldsymbol D} f|\ll \mm \qquad\text{if and only if}\qquad f \in W^{1,1}(\X).\label{eq:RCDW11}
\end{equation}
Moreover, in this case, $|Df|_1= \frac{\d  | {\boldsymbol D} f|}{\d \mm}$ at $\mm$-a.e.\ point. Therefore, building on top of \cite{Deng20},\cite{ACCMCS20} and our Theorem \ref{thm:countable_master_tp}, we are then able to prove:
\begin{theorem}
Let $(\X,\sfd,\mm)$ be a ${\sf RCD}(K,N)$ space with $N<\infty$ and $\mm$ finite. Then, there exists a $\infty$-test plan, denoted by $\pi_{\sf m}$ and concentrated on geodesics, so that:

If $f, G \in L^1(\mm)$ are so that $f\circ \gamma \in W^{1,1}(0,1)$ for $\ppi_{\sf m}$-a.e.\ $\gamma \in AC([0,1],\X)$ and
\begin{equation} 
\Big|\frac{\d }{\d t}f(\gamma_t)\Big| \le G(\gamma_t)|\dot \gamma_t| \qquad (\ppi_{\sf m}\otimes \mathcal{L}_1)\text{-a.e.\ }(\gamma,t), \label{eq:W11master}
\end{equation}
then $f \in  W^{1,1}(\X)$ and $G$ is a $1$-weak upper gradient.
\end{theorem}
\begin{proof}
Since ${\sf RCD}(K,N)$ spaces are non-branching \cite[Theorem 1.3]{Deng20}, we know from Theorem \ref{thm:countable_master_tp} that we can find
a countable collection $\Pi$ of $\infty$-test plans concentrated on geodesics which is a master family for \({\rm BV}(\X)\). An argument
as in the proof of Theorem \ref{thm:master_cw_plan} gives rise,
out of the countable collection $\Pi$, a single $\infty$-test plan
$\ppi_{\sf m}$ which is concentrated on geodesics with length at
most $1$ and satisfying the key property:
\[ \Gamma \text{ is } \ppi_{\sf m}\text{-negligible} \quad \iff\quad \Gamma \text{ is }\ppi\text{-negligible} \quad \forall \ppi \in \Pi, \]
for every Borel $\Gamma \subseteq C([0,1],\X)$. Finally, $f,G \in L^1(\mm)$ satisfy \eqref{eq:W11master} if and only if
\[
\Big|\frac{\d }{\d t}f(\gamma_t)\Big| \le G(\gamma_t)|\dot \gamma_t| \qquad (\ppi\otimes \mathcal{L}_1)\text{-a.e.\ }(\gamma,t), \ \forall \ppi \in \Pi.\]
This implies that for $\ppi$-a.e.\ $\gamma$, it holds  $f \circ \gamma \in{\rm BV}(0,1)$ (in fact, it
is absolutely continuous) with $|D(f\circ \gamma)|(I) \le \int_I G(\gamma_t)|\dot \gamma_t|\, \d t$, for every $I\subseteq [0,1]$ Borel and $\ppi \in \Pi$. Thus, we reach
\[ \begin{split}
\int \gamma_\# |D(f\circ \gamma)| (B)\, \d \ppi(\gamma) &\le  \int\!\!\!\int_0^1 \1_{\gamma^{-1}(B)}(t) G(\gamma_t)|\dot \gamma_t|\, \d t \, \d \ppi(\gamma) \\
&\leq \Lip(\ppi)\int\!\!\!\int_0^1 (\1_B G)\circ \e_t\, \d t \, \d \ppi\\
&\le {\rm Comp}(\ppi) {\rm Lip}(\ppi) \int_B G\, \d \mm,
\end{split}
 \]
 for every $B\subseteq \X$ Borel and $\ppi \in \Pi$. This means that $f \in{\rm BV}_\Pi(\X)$ and $|{\boldsymbol D} f|_\Pi \le G\mm$. By Theorem \ref{thm:countable_master_tp}, this immediately implies that $f \in {\rm BV}(\X)$ with $|{\boldsymbol D}f| \le G\mm$ and, recalling \eqref{eq:RCDW11}, the conclusion.
\end{proof}
\def\cprime{$'$} \def\cprime{$'$}

\end{document}